\RequirePackage[orthodox]{nag}
\PassOptionsToPackage{table,xcdraw}{xcolor}
\makeatletter
\providecommand*{\input@path}{}
\g@addto@macro\input@path{{include/}{../include/}{springer/}}
\makeatother

\newif\ifllncs
\llncsfalse
\newif\ifspringer
\springerfalse

\newif\ifarxiv
\arxivfalse
\arxivtrue

\ifllncs
    \documentclass[oribibl,envcountsame]{include/llncs}
    \title{An Abstract Model for Branch and Cut}
    \author{
        Aleksandr M. Kazachkov\inst{1}\orcidID{0000-0002-4949-9565}
        \and
        Pierre Le Bodic\inst{2}\orcidID{0000-0003-0842-9533}
        \and
        Sriram Sankaranarayanan\inst{3}\orcidID{0000-0002-4662-3241}
    }
    \institute{
        University of Florida, USA \\
        \email{akazachkov@ufl.edu}
        \and
        Monash University \\
        \email{pierre.lebodic@monash.edu}
        \and
        Indian Institute of Management, Ahmedabad \\
        \email{srirams@iima.ac.in}
    }
\fi

\ifspringer
    \RequirePackage{amsmath}
    \documentclass[smallextended,envcountsame,nospthms]{springer/svjour3}
    \journalname{Mathematical Programming}
    \smartqed  %
    \usepackage[misc]{ifsym}
    
    \title{An Abstract Model for Branch and Cut%
        \thanks{Extended version of workshop paper from The 23rd Conference on Integer Programming and Combinatorial Optimization~\cite{KazLeBSan22}.}}
    
    \author{
        Aleksandr M. Kazachkov
        \and
        Pierre Le Bodic
        \and
        Sriram Sankaranarayanan
    }
    
    \institute{%
             \textsuperscript{\href{mailto:akazachkov@ufl.edu}{\Letter}} A. M. Kazachkov
             \at
             University of Florida \\
             \email{akazachkov@ufl.edu} \\
             ORCID: 0000-0002-4949-9565
             \and
             P. Le Bodic
             \at
             Monash University \\
            \email{pierre.lebodic@monash.edu} \\
            ORCID: 0000-0003-0842-9533
            \and
            S. Sankaranarayanan
            \at
            Indian Institute of Management, Ahmedabad \\
            \email{srirams@iima.ac.in} \\
            ORCID: 0000-0002-4662-3241
    }
    \date{\today}
    \authorrunning{Kazachkov, Le Bodic, Sankaranarayanan}
\fi

\ifarxiv
    \documentclass[11pt]{article}
    \usepackage[margin=1in]{geometry}
	\usepackage[noblocks,auth-sc]{authblk}
	
    \title{An Abstract Model for Branch and Cut\thanks{Extended version of workshop paper from The 23rd Conference on Integer Programming and Combinatorial Optimization~\cite{KazLeBSan22}.}}
	\author[1]{Aleksandr M. Kazachkov}
	\author[2]{Pierre Le Bodic}
	\author[3]{Sriram Sankaranarayanan}
	\affil[1]{akazachkov@ufl.edu, University of Florida, Gainesville, FL, USA}
	\affil[2]{pierre.lebodic@monash.edu, Monash University, Clayton, VIC, Australia}
	\affil[3]{srirams@iima.ac.in, Indian Institute of Management, Ahmedabad, India}
    \date{}
\fi

\usepackage[utf8]{inputenc}

\usepackage{amsmath,amssymb,color,enumerate,graphicx,mathrsfs,algorithm,algorithmicx,algpseudocode}
\usepackage[subrefformat=parens]{subcaption}
\usepackage{bbm} %
\usepackage{amsthm}
\usepackage{xspace}
\usepackage[usenames,dvipsnames]{xcolor}
\usepackage[shortlabels,inline]{enumitem}
\setlist{nolistsep}

\usepackage[sort&compress,numbers]{natbib}

\usepackage[colorlinks=true,allcolors=blue]{hyperref}
\usepackage[nameinlink]{cleveref}
\definecolor{mycolor}{RGB}{200,0,0}

\usepackage{tikz}
\usetikzlibrary{calc}
\usepackage{ifthen}
\usetikzlibrary{arrows, arrows.meta, shapes, backgrounds}
\usetikzlibrary{decorations.pathmorphing, patterns} %
\usetikzlibrary{graphs, positioning, fit, quotes}
\usepackage{forest}
\tikzset{itria/.style={
        draw,shape border uses incircle,
        isosceles triangle,shape border rotate=90,yshift=-17pt}}
\tikzset{empty/.style={draw=none,fill=none}}
\tikzset{cut/.style={regular polygon, regular polygon sides=4,minimum size=2.1213mm}}
\tikzset{branch/.style={circle,minimum size=1.5mm}}
\tikzset{big cut/.style={regular polygon, regular polygon sides=4,minimum size=7.0711mm}}
\tikzset{big branch/.style={circle,minimum size=5mm}}
\forestset{small node/.style={inner sep=0pt, branch, fill=black, draw, font=\small},
    l=10pt,->}
\forestset{big node/.style={inner sep=0pt, big branch, draw, font=\small},
    l=10pt,->}
\forestset{
    join aunts/.style={
    before drawing tree={
      tempkeylista'=,
      for nodewalk={fake=u, siblings}{tempkeylista/.option=name},
      join list/.register=tempkeylista,
      tikz+/.process={ OOw2 {join list} {fork sep} { \draw  [thick, rounded corners, Stealth-]\foreach \i in {##1} { (.child anchor) -- ++(0,##2) -| (\i.parent anchor) }; } },
    }
  }
}
\tikzset{
  subtreewt/.append  style={align=center,
    label={[fill=none,font=\small,yshift=-0.25cm]0:{\hspace{1.5pt} subtree \heavycutweight = #1}}},
}
\tikzset{
  subtreewtdeep/.append  style={align=center,
    label={[fill=none,font=\small,yshift=-0.45cm]0:{\hspace{2.5pt} subtree \heavycutweight = #1}}},
}
\tikzset{
  lvlwt/.append  style={align=center,
    label={[fill=none,font=\small]0:{\hspace{1.5pt} layer \heavycutweight = #1}}},
}
\tikzset{
  nodewt/.style={align=center,
    label={[fill=none,font=\small]0:{\hspace{1.5pt} \heavycutweight = #1}}},
}

\makeatletter
\DeclareRobustCommand{\rvdots}{%
  \vbox{
    \baselineskip4\p@\lineskiplimit\z@
    \kern-\p@
    \hbox{.}\hbox{.}\hbox{.}
  }}
\makeatother
\makeatletter
\def\Ddots{\mathinner{\mkern1mu\raise\p@
\vbox{\kern7\p@\hbox{.}}\mkern2mu
\raise4\p@\hbox{.}\mkern2mu\raise7\p@\hbox{.}\mkern1mu}}
\makeatother
\forestset{dot node/.style={big node,font={$\rvdots$},draw=none,fill=none}}

\newcommand{\dotsangle}{26}
\newcommand{\dotsxshift}{.2ex}
\newcommand{\dotsyshift}{1ex}
\newcommand{\Ldots}{\hspace{\dotsxshift}%
    \raisebox{\dotsyshift}{\rotatebox{\dotsangle}{$\Ddots$}}}

\forestset{ldiag dot node/.style={big node,font={$\Ldots$},draw=none,fill=none}}
\forestset{rdiag dot node/.style={big node,font={$\ddots$},draw=none,fill=none}}

\newcommand{\exqed}{\hfill$\blacksquare$}
\renewcommand{\tilde}{\widetilde}
\renewcommand{\hat}{\widehat}

\newcommand*{\defeq}{\mathrel{\vcenter{\baselineskip0.5ex \lineskiplimit0pt
                     \hbox{\footnotesize.}\hbox{\footnotesize.}}}%
                     =}

\colorlet{rosso}{red!80!blue}

\usepackage[normalem]{ulem}

\newcommand{\cardinality}[1]{\lvert #1 \rvert}

\ifllncs
\spnewtheorem{condition}[theorem]{Condition}{\bfseries}{\itshape}
\else
\newtheorem{example}{Example}
\newtheorem{lemma}[example]{Lemma}
\newtheorem*{lemma*}{Lemma}
\newtheorem{theorem}[example]{Theorem}
\newtheorem*{theorem*}{Theorem}
\newtheorem{corollary}[example]{Corollary}

\newtheorem{definition}[example]{Definition}

\newtheorem*{claim*}{Claim}
\theoremstyle{remark}

\fi

\allowdisplaybreaks

\crefname{lemma}{Lemma}{Lemmas}
\crefname{theorem}{Theorem}{Theorems}
\crefname{claim}{Claim}{Claims}
\crefname{algorithm}{Algorithm}{Algorithms}
\crefname{definition}{Definition}{Definitions}
\crefname{assume}{Assumption}{Assumptions}
\crefname{condition}{Condition}{Conditions}
\crefname{remark}{Remark}{Remarks}
\crefname{example}{Example}{Examples}
\crefname{figure}{Figure}{Figures}
\crefname{section}{Section}{Sections}
\crefname{table}{Table}{Tables}
\crefname{enumi}{Statement}{Statements}
\crefname{line}{Step}{Steps}
\crefname{equation}{}{}

\DeclareMathOperator*{\argmin}{arg\,min}
\DeclareMathOperator*{\argmax}{arg\,max}
\newcommand{\suchthat}{:}
\newcommand{\R}{\mathbb{R}}

\newcommand{\Z}{\mathbb{Z}}
\newcommand{\reals}{\R}
\newcommand{\integers}{\Z}
\newcommand{\nonnegreals}{\reals_{\scriptscriptstyle \ge 0}}
\newcommand{\nonnegints}{\integers_{\scriptscriptstyle \ge 0}}

\newcommand{\posints}{\integers_{\scriptscriptstyle > 0}}

\newcommand{\bb}{B\&B\xspace}
\newcommand{\bc}{B\&C\xspace}

\renewcommand{\bar}{\overline}
\renewcommand{\hat}{\widehat}
\renewcommand{\tilde}{\widetilde}

\newcommand{\cutind}{t}

\newcommand{\HCf}{\operatorname{w}}
\newcommand{\treeweight}{\operatorname{\tau}}
\newcommand{\treetime}{\treeweight}
\newcommand{\treesize}{\treeweight}

\newcommand{\heavycutweight}{time\xspace}

\newcommand{\floor}[1]{\left\lfloor#1\right\rfloor}
\newcommand{\ceil}[1]{\left\lceil#1\right\rceil}

\newcommand{\depth}{\delta}
\newcommand{\gap}{z}
\newcommand{\lbd}{\ell}
\newcommand{\rbd}{r}
\newcommand{\cutbd}{c}
\newcommand{\cbd}{\cutbd}
\newcommand{\gapfn}[1][]{\gap_{\scriptscriptstyle #1}}

\newcommand{\targetbound}{Z}

\newcommand{\SVBC}{\text{SVBC}\xspace}
\newcommand{\allones}{\mathbf{1}}
\newcommand{\lrc}[1][\lbd,\rbd;\cutbd,\allones]{(#1)}
\newcommand{\lrcw}[1][\lbd,\rbd;\cutbd,\HCf]{(#1)}
\newcommand{\svbc}[1][\lrcw]{\mathop{SV\!BC}#1}
\newcommand{\SVBWCText}{Single Variable Branching with Harmonically-Worsening Cuts}
\newcommand{\SVBWC}{\text{SVBHC}\xspace}
\newcommand{\svbwc}[1][\lrcw]{\mathop{SV\!BH\!C}#1}
\newcommand{\SVBHC}{\text{SVBHC}\xspace}

\newcommand{\nodes}{\mathcal{V}}

\newcommand{\MOP}{OP}

\newcommand{\ncuts}{\kappa}
\newcommand{\ncutsZ}[1][\targetbound]{\kappa_{\scriptscriptstyle #1}}
\newcommand{\ncutsvbwc}{\bar{\ncuts}}
\newcommand{\depthmax}{\depth^{\max}}

\newcommand{\timefunction}{time-function}
\newcommand{\Timefunction}{Time-function}

\newcommand{\lbFun}[1][\phi]{\underline{#1}}
\newcommand{\ubFun}[1][\phi]{\bar{#1}}

\newcommand{\lbFunOptCont}{\underline{\depth}^c}
\newcommand{\ubFunOptCont}{\bar{\depth}^c}
\newcommand{\lbFunOpt}{\underline{\depth}}
\newcommand{\ubFunOpt}{\bar{\depth}}

\DeclareMathOperator{\deepestcutnodestart}{last\_b\_then\_c}
\DeclareMathOperator{\distance}{dist}

\DeclareMathOperator{\lengthcutpath}{ncuts}

\newcommand{\branchandcut}{branch and cut}
\newcommand{\Branchandcut}{Branch and cut}
\newcommand{\branchandbound}{branch and bound}

\newcommand{\branchandcutx}{branch-and-cut}

\newcommand{\branchandboundx}{branch-and-bound}

\newcommand{\instance}[1]{\texttt{#1}}

\begin{document}

\maketitle

\begin{abstract}
    \Branchandcut{} is the dominant paradigm for solving a wide range of mathematical programming problems --- linear or nonlinear --- combining efficient search (via \branchandbound{}) and relaxation-tightening procedures (via cutting planes, or cuts).
    While there is a wealth of computational experience behind existing cutting strategies, there is simultaneously a relative lack of theoretical explanations for these choices, and for the tradeoffs involved therein.
    Recent papers have explored abstract models for branching and for comparing cuts with \branchandbound{}.
    However, to model practice, it is crucial to understand the impact of jointly considering branching and cutting decisions.
    In this paper, we provide a framework for analyzing how cuts affect the size of \branchandcutx{} trees, as well as their impact on solution time.
    Our abstract model captures some of the key characteristics of real-world phenomena in \branchandcutx{} experiments,
    regarding whether to generate cuts only at the root or throughout the tree, how many rounds of cuts to add before starting to branch, and why cuts seem to exhibit nonmonotonic effects on the solution process.
    \ifspringer
    \fi
\end{abstract}

\section{Introduction}

The \branchandcutx{} (\bc) paradigm is a hybrid of the \branchandboundx{} (\bb) \cite{LanDoi60}
and cutting plane methods~\cite{Gomory58,Gomory60_gmic,Gomory63}.
It is central to a wide range of modern global optimization approaches~\citep{burer2008quadbrach,al1987lcpbranch}, particularly mixed-integer linear and nonlinear programming solvers~\citep{50yearsbook}.
Cutting planes, or \emph{cuts}, tighten the relaxation of a given optimization problem and are experimentally known to significantly improve a \bb process~\cite{AchWun13}, but determining which cuts to add is currently based on highly-engineered criteria and computational insights, not from theory.
An outstanding open problem is a rigorous underpinning for the choices involved in \branchandcut{}.
While recent papers have been actively exploring the theory of branching~\citep{leBodic2017,AndLeBMor20,DeyDubMol21,DeyDubMol22,DeyDubMolSha21+} and comparing cutting and branching~\citep{BasConSumJia23_complexity-of-bb-and-cuts-I}, the interaction of the two together remains poorly understood.
Most recently, \citet*{BasConSumJia22_complexity-of-bb-and-cuts-II} have proved that using \bc can strictly outperform either branching or cutting alone.

This paper introduces a theoretical framework for analyzing the practical challenges involved in making \bc decisions.
We build on work by \citet*{leBodic2017}, which provides an abstract model of \bb, based on how much bound improvement is gained by branching on a variable at a node of the \bb tree.
This model not only is theoretically useful, but also can improve branching decisions in solvers~\citep{AndLeBMor20}.

Specifically, we add a cuts component to the abstract \bb model from \citet*{leBodic2017}.
We apply this enhanced model to account for both the utility of the cuts in proving bounds, as well as the additional time taken to solve the nodes of a \bc tree after adding cuts. 

In this abstract model, given the relative strengths of cuts, branching, and the rate at which node-processing time grows with additional cuts, we quantify (i) the number of cuts, and (ii) cut positioning (at the root or deeper in the tree) to minimize both the tree size and the solution time of an instance.
This thereby captures some of the main tradeoffs between cutting and branching, in that cuts can improve the bound or even the size of a \bc tree, but meanwhile slow down the solution time overall.
We use a \emph{single-variable} abstract \bc model, where every branching variable has identical effect on the bound, and we only address the \emph{dual} side of the problem, i.e., we are only interested in proving a good \emph{bound} on the optimal value, as opposed to generating better integer-feasible solutions.

We emphasize that our motivation is to advance a theoretical understanding of empirically-observed phenomena in solving optimization problems, and our results show that some of the same challenges that solvers encounter in applying cuts do arise in theory.
While we state prescriptive recommendations in our abstract model, these are not intended to be immediately computationally viable.
Instead, the intent of the prescriptive results is to see whether our abstraction affords enough simplicity to make precise theoretical statements.

\paragraph{Summary of contributions and paper structure.}

We provide a generic view of \bc in \Cref{sec:notation}.
\Cref{sec:model} introduces our abstract \bc model, in which the quality of cuts and branching remains fixed throughout the tree.
In \Cref{sec:tree-size}, we analyze the effect of cuts on tree size; we prove that in this case it is never necessary to add cuts after the root node, and we provide a lower bound on the optimal number of cutting plane rounds that will minimize the \bc tree size.
In \Cref{sec:svbhc}, we extend our model to account for diminishing marginal returns from cuts, relaxing our assumption of constant cut strength.
Our main result in this section is an approximation of the optimal number of cuts.

Then, in \Cref{sec:general-cut-time}, we study how cuts affect solving \emph{time} for a tree, not just its size, under constant cut strength.
In \Cref{sec:cut-time-bounded-by-polynomial}, we show that cuts are guaranteed to be helpful for sufficiently hard instances.
In contrast to the case of tree size, in this more general setting, adding cuts after the root node may be better.
However, in \Cref{thm:root-cuts-suffice}, we show that when the two branching directions yield the same bound improvement, then root cuts are still sufficient.

\section{Preliminaries}
\label{sec:notation}

We are given a generic optimization problem (\emph{\MOP}) --- linear or nonlinear, with or without integers --- which is to be solved using a \bc algorithm. For convenience, we assume that the \MOP{} is a minimization problem.
We also assume that we already have a feasible solution to the \MOP{}, so that our only goal is to efficiently certify the optimality or quality of that solution.

The \bc approach involves creating a computationally tractable relaxation of the original problem, which we call the \emph{root} of the \bc tree and assume is provided to us.
For example, when the \MOP{} is a mixed-integer linear program, we start with its linear programming relaxation.
The value of the solution to this relaxation provides a lower bound on the optimal value to the \MOP{}.
\bc{} proceeds by either
  (1) tightening the relaxation through adding \emph{valid} cuts, which will remove parts of the current relaxation but no \MOP{}-feasible points, or
  (2) splitting the feasible region, creating two subproblems, which we call the \emph{children} of the original (\emph{parent}) relaxation.
Both of these operations improve the lower bound 
with respect to the original relaxation.
The \bc{} procedure repeats on the new relaxation with cuts added in the case of (1), and recursively on the children in the case of (2); we assume that tractability is maintained in either case.
Moreover, we assume that all children remain \MOP{}-feasible.
We now formally define a \emph{\bc tree} as used in this paper.

\begin{definition}[\bc tree]
\label{def:bctree}
A \bc tree $T$ is a rooted binary tree with node set $\nodes_T$ that is node-labeled by a function $\gapfn[T]: \nodes_T \to \nonnegreals$, indicating the bound improvement at each node with respect to the bound at the root node, such that
\begin{enumerate}
    \item The \emph{root} node $v_0$ has label $\gapfn[T](v_0)=0$.
    \item
      A node $v$ with exactly one child $v'$ is a \emph{cut} node, and we say that a \emph{cut} or \emph{round of cuts} is added at node $v$.
      The bound at $v'$ is $\gapfn[T](v') = \gapfn[T](v) + \cutbd_v$, where $\cutbd_v$ is the nonnegative value associated with the round of cuts at $v$.
    \item
      A node $v$ with exactly two children $v_1$ and $v_2$ is a \emph{branch} node, and we say that we \emph{branch} at node $v$.
      The bounds at the children of $v$ are $\gapfn[T](v_1) = \gapfn[T](v) + \lbd_v$ and $\gapfn[T](v_2) = \gapfn[T](v) + \rbd_v$, where $(\lbd_v, \rbd_v)$ is the pair of bound improvement values associated with branching at $v$.
    \item
      A node with no children is a \emph{leaf} node.
\end{enumerate}
We say that $T$ \emph{proves a bound} of $\targetbound$ if $\gapfn[T](v) \ge \targetbound$ for all leaves $v \in \nodes_T$.
\end{definition}

We will refer to a \emph{cut-and-branch} tree as one in which all cut nodes are at the root, before the first branch node.

While \Cref{def:bctree} is generic, the abstraction we study is restricted to the \emph{single-variable} version in which $\lbd_v$ and $\rbd_v$ are the same for each branch node $v \in \nodes_T$.
We also drop the subscript $v$ in $\cutbd_v$, as in \Cref{sec:tree-size} and \Cref{sec:general-cut-time}, we assume a constant cut quality for each cut node $v \in \nodes_T$, while in \Cref{sec:svbhc}, cut quality is only a function of the number of cuts already applied.

\section{The Abstract Branch-and-Cut Model}
\label{sec:model}
This section introduces the \emph{Single Variable Branch-and-Cut}~(\SVBC) model, an abstraction of a \bc tree as presented in \Cref{def:bctree}.
First, we define a formal notion of the time taken to process a \bc tree as the sum of the node processing times, which in turn depends on the following definition of a \timefunction{}.
\begin{definition}[\Timefunction{}]
\label{def:heavy-cut-fn}
A function $\HCf:\nonnegints\to[1,\infty)$ is a \emph{\timefunction{}} if it is nondecreasing and $\HCf(0) = 1$.
\end{definition}

\begin{definition}[Node time and tree time]
\label{def:node-and-tree-time}
Given a \bc tree $T$, node $v \in \nodes_T$, and \timefunction{} $\HCf$,
\begin{enumerate}[(i)]
    \item 
    the \emph{(node) time} of $v$, representing the time taken to process node $v$, is 
    $\HCf(z)$,
    where $z$ is the number of \emph{cut} nodes in the path from the root of $T$ to $v$.
    \item the \emph{(tree) time} of $T$, denoted by $\treeweight_{\HCf}(T)$, is the sum of the node times of all the nodes in the tree.
\end{enumerate}
We simply say $\treeweight(T)$ when the \timefunction{} $\HCf$ is clear from context.
\end{definition}
Definition~\ref{def:node-and-tree-time} models the observation that cuts generally make the relaxation harder to solve, and hence applying more cuts increases node processing time.
Note that 
\begin{enumerate*}[(i)]
\item if $\HCf = \allones$, i.e., $\HCf(z) = 1$ for all $z\in\nonnegints$, we obtain the regular notion of size of a tree, which counts the number of nodes in the tree, and
\item the \heavycutweight of a pure cutting tree with $t$ cuts (i.e., $t+1$ nodes) is $\sum_{i=0}^{t} \HCf(i)$.
\end{enumerate*}

Finally, we state the \SVBC model in \Cref{def:svbc}.
In this model, the relative bound improvement at every cut node is always the same constant $\cutbd$, and every branch node is associated to the same $(\lbd,\rbd)$ pair of bound improvement values.
We also assume that the time to solve a node depends on the number of cuts added to the relaxation up to that node.

\begin{definition}[Single Variable Branch-and-Cut (\SVBC) Tree]
\label{def:svbc}
A \bc tree is a \emph{Single Variable Branch-and-Cut (\SVBC)} tree with parameters $\lrcw[\lbd,\rbd;\cutbd,\HCf]$ if the bound improvement value associated with each branch node is $(\lbd, \rbd)$, the bound improvement by each cut node is $\cutbd$, and the \timefunction{} is $\HCf$.
We say such a tree is an $\svbc$ tree.
\end{definition}

\noindent Without loss of generality, we assume $0 \le \lbd \le \rbd$. 

\begin{definition}[$\treeweight$-minimality]
Given a function $\HCf:\nonnegints\to[1,\infty)$, we say that a \bc tree $T$ that proves bound $\targetbound$ is $\treeweight$-minimal if, for any other \bc tree $T'$ that also proves bound $\targetbound$ with the same $\lrcw$, it holds that $\treeweight(T') \ge \treeweight(T)$.
\end{definition}
\noindent
When $\HCf = \allones$, we may refer to a $\treeweight$-minimal tree as \emph{minimal-sized}.

It is often the case that applying a round of cuts at a node may not improve the bound as much as branching at that node, but the advantage is that cutting adds only one node to the tree, while branching creates two subproblems.
A first question is whether there always exists a minimal-size tree with \emph{only} branch nodes or \emph{only} cut nodes.
We address this in \cref{ex:branch-and-cut}, which illustrates our notation, shows that cut nodes can help reduce the size of a \bc{} tree despite improving the bound less than branch nodes, and highlights the fact that finding a minimal-sized \bc{} tree proving a particular bound $\targetbound$ involves strategically using both branching and cutting.

\begin{example}[\Branchandcut{} can outperform pure branching or pure cutting]
\label{ex:branch-and-cut}
\begin{figure}[ht]
    \centering
    \captionsetup[subfigure]{justification=centering}
    \begin{subfigure}[b]{0.33\textwidth}
    \centering
    \begin{forest}
        for tree = {big node}
        [0
            [3
                [6] [6]
            ]
            [3
                [6] [6]
            ]
        ]
    \end{forest}
    \caption{Pure branching: 7 nodes}
    \label{fig:BCbetter:a}
    \end{subfigure}
    \begin{subfigure}[b]{0.31\textwidth}
    \centering
    \begin{forest}
        for tree = {big node}
        [0,big cut [1,big cut [,dot node [5,big cut [6]]]]]
    \end{forest}
    \caption{Pure cutting: 7 nodes}
        \label{fig:BCbetter:b}
    \end{subfigure}
    \begin{subfigure}[b]{0.33\textwidth}
    \centering
    \begin{forest}
        for tree = {big node}
        [0,big cut [1,big cut [2,big cut [3 [6] [6] ] ] ] ]
    \end{forest}
    \caption{\Branchandcut{}: 6 nodes}
        \label{fig:BCbetter:c}
    \end{subfigure}
    \caption{%
      Three \bc trees proving $\targetbound = 6$, with $\lbd = \rbd = 3$, $\cutbd = 1$, and $\HCf = \allones$.
    }
    \label{fig:BCbetter}
\end{figure}
  \Cref{fig:BCbetter} shows three \bc trees that prove the bound $\targetbound = 6$.
  The tree in panel~\subref{fig:BCbetter:a} only has branch nodes,
  \subref{fig:BCbetter:b} only has cut nodes,
  and \subref{fig:BCbetter:c} has both branch and cut nodes.
  As seen in the figure, branching and cutting together can create strictly smaller trees than pure branching or cutting methods. \exqed
\end{example}

\citet*{BasConSumJia23_complexity-of-bb-and-cuts-I,BasConSumJia22_complexity-of-bb-and-cuts-II} also investigate the complementary effect of branching and cutting.
The authors prove that for pure binary problems, when cutting and branching are derived from the same underlying logical conditions, then it suffices to \emph{only} cut to minimize the size of the tree~\cite{BasConSumJia23_complexity-of-bb-and-cuts-I}.
When the second assumption is relaxed, the second paper proves that combining cutting and branching can be exponentially better than using either method alone~\cite{BasConSumJia22_complexity-of-bb-and-cuts-II}.
We instead focus on specifying the optimal number of cuts to add or where to place them in the tree for a particular instance.

\section{Optimizing Tree Size}
\label{sec:tree-size}

In this section, we examine 
the number of cuts that minimize the \emph{size} $\cardinality{\nodes_T}$ of a \bc tree $T$, i.e., optimizing $\treeweight(T)$ when $\HCf = \allones$.
In \cref{lem:cutFirst}, we first address the location of these cuts --- should they be at the root or deeper in the tree?

\begin{lemma}\label{lem:cutFirst}
For any target bound $\targetbound$ and a fixed set of parameters $\lrc$, 
there exists a $\treeweight$-minimal $\svbc[\lrc]$ tree
that proves bound $\targetbound$ such that all cut nodes form a path starting at the root of the tree.
\end{lemma}
\begin{proof}
Let $T$ be a $\treeweight$-minimal $\svbc[\lrc]$ tree.
If all cut nodes in the tree $T$ are at the root, then we are done.
Otherwise, let $v \in \nodes_T$ be a cut node with a parent that is a branch node, i.e., $v$ has one child $w$.
Let $T'$ be the tree obtained by removing $w$ from $T$, i.e., contracting $v$ and $w$, and instead inserting $w$ immediately after the root.
Let $v' \ne w$ be a leaf node of $T$, which is also a leaf of $T'$.
It holds that $\gapfn[T'](v') \ge \gapfn[T](v')$, since the path from the root to $v'$ in $T'$ goes through the same branch nodes and at least as many cut nodes as in $T$.
Recursively applying this procedure, we move all cut nodes to the root without increasing the tree size, proving the desired result by the assumed minimality of $T$.
\end{proof}

We have proved that for any minimal-size $\svbc[\lrc]$ tree, it suffices to consider cut-and-branch trees, where all cut nodes are at the root.
To understand how \emph{many} cuts should be added,
we start with the special case that $\cbd \le \lbd = \rbd$.

A useful observation for our analysis is that one should not evaluate the effects of cuts one at a time on the size of the tree, as tree size does not monotonically decrease as the number of cuts increases from $0$ to the optimal number of cuts.
For example, if $\cutbd < \lbd = \rbd$ and $\targetbound = 2\cutbd \pmod r$, using one cut node would \emph{increase} the overall tree size, while two cut rounds would reduce tree size by $2^{\ceil{\targetbound/\rbd}}-2$.
This phenomenon highlights a practical challenge in determining how to use a cut family and whether cuts benefit an instance, as adding too few or too many cut nodes may increase tree size while the right number can greatly decrease the overall size.

Instead, the key insight for \Cref{thm:SVBcutGood_lequalsr} is reasoning about \emph{layers}: adding a \emph{set} of cut nodes is beneficial when, together, the cut nodes improve the bound enough to remove an additional layer of the \branchandboundx{} tree, and fewer cuts are added than the number of removed nodes.

If a minimal-size tree $T$ proving bound $\targetbound$ has $k$ cut nodes at the root, then the depth of the \emph{branching component}, the subtree starting with the first branch node, is
    $\depth_k \defeq \max \{0, \ceil{(\targetbound - \cbd k) / \rbd} \}$.
The total size of the tree is
    $\treetime(T) = k + 2^{\depth_k + 1} - 1$.
We also know that the depth of the branching component when the target bound is $\targetbound$ is never more than
    $
        \depthmax \defeq \ceil{\targetbound / \rbd}.
    $

For any given $\depth \in \{0,\ldots,\depthmax\}$ and target bound $\targetbound$,
the minimum number of cut nodes at the root to achieve that depth of the branching component is 
    \[
        \ncutsZ(\depth) \defeq \max\{0,\ceil{(\targetbound - \depth \rbd)/\cutbd}\},
    \]
where it can be seen that $\ncutsZ(\depth) = 0$ if and only if $\depth = \depthmax$, within the domain.

\begin{lemma}
\label{lem:svbc-opt-to-cut-layers}
    When $0 < \cutbd \le \lbd = \rbd$, 
	the optimal number of cut nodes in a minimal-size \SVBC tree proving bound $\targetbound$ is $\ncutsZ(\depth)$ for some $\depth \in \nonnegints$.
\end{lemma}
\begin{proof}
    A branching component with depth $\depth$ proves a bound $\depth \rbd$, leaving a bound of $\max \{ 0, \targetbound - \depth \rbd \}$ to prove with cut nodes.
    Therefore, it is necessary and sufficient to use 
        $\ncutsZ(\depth) = \max \left\{ 0,\left \lceil {(\targetbound - \depth \rbd)}/{\cbd} \right \rceil \right\}$ 
    cut nodes.
\end{proof}

Next, we present \cref{thm:SVBcutGood_lequalsr}, which provides the optimal number of rounds of cuts for an $\svbc[\lrc]$ tree when $\cutbd \le \lbd = \rbd$, as a function of the tree parameters and the target bound.
The theorem implies that the depth of the branching component in a minimal-size tree can take one of four values, and it is at most $\depth^* \defeq \floor{\log_2 \ceil{{\rbd}/{\cutbd}}}$,
which is independent of the target bound $\targetbound$.
Thus, as $\targetbound$ increases, the proportion of the bound proved by branch nodes goes to zero.

\begin{theorem}
\label{thm:SVBcutGood_lequalsr}
  Let $\depth^* \defeq \floor{\log_2 \ceil{{\rbd}/{\cutbd}}}$.
    When $0 < \cutbd \le \lbd = \rbd$, the number of cut nodes to minimize the size of an $\svbc[\lrc]$ tree proving bound $\targetbound$ is
    \begin{equation*}
    k^* \defeq
    \begin{cases}
        \ncutsZ(\depth^*)
        &\text{if $\targetbound \ge \rbd \depth^*$ 
            and 
            $
                \ncutsZ(\depth^*-1) -  \ncutsZ(\depth^*)
                \ge 2^{\depth^*}
            $}
        \\
        \ncutsZ(\depth^* - 1)
        &\text{if $\targetbound \ge \rbd \depth^*$ 
            and 
            $
                \ncutsZ(\depth^*-1) -  \ncutsZ(\depth^*)
            < 2^{\depth^*}
            $}
        \\
        \ncutsZ(\depthmax - 1)
        &\text{if $\targetbound < \rbd \depth^*$ 
            and
            $
                \ncutsZ(\depthmax-1)
            < 2^{\depthmax}
            $}
        \\
        0 & \text{otherwise.}
    \end{cases}
    \end{equation*}
    The optimal depth of the branching component is either $\min\{\depth^*,\depthmax\}$ or $\min\{\depth^*,\depthmax\} - 1$.
    Moreover, the size of any minimal $\svbc[\lrc]$ tree that proves bound $\targetbound$ is at least
    \(
      2^{\ceil{(\targetbound - \cbd k^*)/\rbd}+1}-1+k^*.
    \)
\end{theorem}
\begin{proof}

    Given an instance for which bound $\targetbound$ needs to be proved, our goal is to understand how the size of the $\SVBC{\lrc}$ tree changes as a function of $k$, the number of cuts we apply at the root node.
	By \Cref{lem:svbc-opt-to-cut-layers},
	our goal is equivalent to finding the optimal depth of the branching component.
    
    Let $T_{\depth}$ denote the tree with $\ncutsZ(\depth)$ cuts added at the root node, followed by a branching component of depth $\depth$.
    Recall that $\ncutsZ(\depth)$ is the minimum number of cuts to achieve a branching depth of $\depth$.
    The bound $\gapfn[T_{\depth}](u)$ at each leaf node $u$ of $T_{\depth}$ satisfies
    $\gapfn[T_{\depth}](u) \ge \targetbound$.
    Hence, for any node $v$ that is a parent of a leaf node $u$ of $T_{\depth}$, 
    the bound at $v$ is $\gapfn[T_{\depth}](v) = \gapfn[T_{\depth}](u) - \rbd \ge \targetbound - \rbd$.
    By definition of $\ncutsZ(\depth-1)$,
    $\gapfn[T_{\depth}](v) + (\ncutsZ(\depth-1) - \ncutsZ(\depth)) \cbd \ge \targetbound$,
    as the last layer of the tree $T_\depth$ will no longer be necessary,
    and any fewer cuts will not meet the target bound:
        \[
            \gapfn[T_{\depth}](v) + (\ncutsZ(\depth-1) - \ncutsZ(\depth) - 1) \cbd 
            < \targetbound.
        \]
    Hence, %
        \begin{equation*}
        \label{eq:upper-bound-on-ncuts}
            \ncutsZ(\depth - 1) - \ncutsZ(\depth)
            < 1 + \frac{\targetbound - \gapfn[T_{\depth}](v)}{\cbd}
            \le 1 + \frac{\rbd}{\cbd},
        \end{equation*}
    so that the number of cuts to decrease the branching component by one more layer is at most $\ncutsZ(\depth - 1) - \ncutsZ(\depth) \le \floor{\rbd/\cbd}$.
    As there are $2^\depth$ leaf nodes in the last layer of $T_\depth$,
    $\treeweight(T_{\depth-1}) > \treeweight(T_{\depth})$
    if and only if $\ncutsZ(\depth - 1) - \ncutsZ(\depth) > 2^{\depth}$,
    implying that adding the $\ncutsZ(\depth - 1) - \ncutsZ(\depth)$ cut nodes
    is beneficial if
        $\depth > \depth^* = \floor{\log_2\ceil{{\rbd}/{\cutbd}}}$.
    This is independent of $\targetbound$ and we conclude that,
    if $\depth^* \le \depthmax$,
    then the optimal branching depth is at most $\depth^*$.
        
    Now assume $\depth < \depthmax$.
    For a leaf node $u$ of $T_{\depth}$,
        $\gapfn[T_{\depth}](u) < \targetbound + \cbd$,
    as the definition of $\ncutsZ(\depth)$ means $\ncutsZ(\depth) - 1$ cut nodes would require another layer of branching to prove bound $\targetbound$.
    For any node $v$ that is a parent of $u$,
    by definition of $\ncutsZ(\depth-1)$,
        $
            \gapfn[T_{\depth}](v) + (\ncutsZ(\depth-1) - \ncutsZ(\depth)) \cbd \ge \targetbound,
        $
    which, together with $\gapfn[T_{\depth}](v) < \targetbound - \rbd + \cbd$, implies that, when $\depth < \depthmax$, the number of cuts to decrease the branching component by one more layer is at least
        \begin{equation}
        \label{eq:lower-bound-on-ncuts}
            \ncutsZ(\depth - 1) - \ncutsZ(\depth)
            \ge \frac{\targetbound - \gapfn[T_{\depth}](v)}{\cbd}
            \ge \ceil{\frac{\rbd}{\cbd}} - 1.
        \end{equation}
    It follows that removing an additional layer weakly increases the size of the tree if
        $\ceil{\rbd/\cbd} - 1 \ge 2^{\depth}$,
    which holds if $\depth < \depth^* - 1$, 
    and so the optimal branching depth is at least $\depth^*-1$.
    Hence, the optimal number of cuts when $\depth^* \le \depthmax$ is $\ncutsZ(\depth^*)$ if $\ncutsZ(\depth^*-1) - \ncutsZ(\depth^*) \ge 2^{\depth^*}$, 
    and it is $\ncutsZ(\depth^*-1)$ otherwise.
    
    The last case to consider is when $\depth^* > \depthmax$, or equivalently $\targetbound < \rbd \depth^*$.
    For all $\depth \le \depthmax \le \depth^*-1$, including the pure \branchandboundx{} tree, $T_{\depth}$ has at most $2^{\depth^*-1} \le \ceil{\rbd/\cbd} - 1$ leaf nodes.
    For depths $\depth < \depthmax$, the lower bound in \eqref{eq:lower-bound-on-ncuts} implies that $\treetime(T_{\depth-1}) \ge \treetime(T_{\depth})$.
    However, at $\depth = \depthmax$, it is still possible that $\ncutsZ(\depthmax-1) < 2^{\depthmax}$.
    This precisely results in the last two cases in the definition of $k^*$ in the theorem statement.
\end{proof}

In \cref{thm:SVBcutGood}, we show that even in general, for $\lbd \ne \rbd$, it is always optimal to add at least one cut round for sufficiently large target bounds.

\begin{theorem}
\label{thm:SVBcutGood}
  If $0 < \cutbd \le \lbd \le \rbd$ and $\targetbound > \rbd\floor{\log_2 \ceil{\rbd/\cutbd}}$, then the minimal $\svbc[\lrc]$ tree proving a bound $\targetbound$ has at least one cut node.
\end{theorem}
\begin{proof}
  Consider a pure branching tree that proves $\targetbound$.
  The number of leaf nodes of this tree is at least $2^{\ceil{\targetbound/\rbd}}$, since $\lbd \le \rbd$, and all the parents of each of these leaf nodes have a remaining bound in $(0,\rbd]$ that needs to be proved.
  Now suppose we add $\ceil{\rbd/\cutbd}$ rounds of cuts.
  All of the leaf nodes of the pure \branchandboundx{} tree would then be pruned, since the parent nodes would already prove the desired target bound of $\targetbound$.
  As a result, 
  there is benefit to cutting when 
    $2^{\ceil{\targetbound/\rbd}} > \ceil{\rbd/\cutbd}$,
  which holds when
    $\targetbound > \rbd\floor{\log_2 \ceil{\rbd/\cutbd}}$.
\end{proof}

\begin{corollary}
\label{cor:SVBcutGood}
  If $0 < \cutbd \le \lbd \le \rbd$, then for $\bar{\targetbound} \defeq \rbd\floor{\log_2 \ceil{\rbd/\cutbd}}$, \emph{every} minimal $\svbc[\lrc]$ tree
  proving a bound $\targetbound > \bar{\targetbound}$  has at least $\ceil{\frac{\targetbound-\bar \targetbound}{\cbd}}$ cut nodes.
\end{corollary}

\begin{example}
\label{ex:independent-set}
  The following example, from \citet*{BasConSumJia22_complexity-of-bb-and-cuts-II}, shows that \SVBC trees with constant $\cutbd \le \lbd = \rbd$ have been studied in the literature and that cuts not only decrease the size of a \branchandboundx{} tree, but in fact can lead to an exponential improvement.
  
  Consider the independent set problem, defined on a graph $G$ with vertices $V$ and edge set $E$, in which $G$ consists of $m$ disjoint triangles (cliques of size three):
    $
      \max_x \{ \sum_{v \in V} x_v : x \in \{0,1\}^{\cardinality{V}};\; x_u + x_v \le 1, \; \forall \,\{u,v\} \in E \}.
    $
  The optimal value is $m$, using $x_v = 1$ for exactly one vertex of every clique, while the linear relaxation has optimal value $3m/2$, obtained by setting $x_v = 1/2$ for all $v \in V$.

  Suppose we branch on $x_v$, $v \in V$, where $v$ belongs to a clique with vertices $u$ and $w$.
  In the left ($x_v \le 0$) branch, the objective value of the relaxation decreases by $\lbd = 1/2$ with respect to the parent.
  This is because the optimal values of the variables for all vertices except $u$, $v$, and $w$ remain unchanged, and the constraint $x_u + x_w \le 1$ along with $x_v \le 0$ implies that the objective contribution of the triangle $\{u,v,w\}$ is at most $1$, whereas at the parent node $x_u + x_v + x_w$ contributed 3/2 to the objective.
  We can attain that contribution of $1$ by setting either $x_u = 1$ and $x_w = 0$, or $x_w = 1$ and $x_u = 0$.
  Similarly, for the right branch, we can derive that $\rbd = 1/2$.

  Notice that once we branch on $x_v$, the remaining problem can be seen as fixing the values of the three variables corresponding to vertices in the triangle that $v$ belongs to, while keeping the remaining variables unchanged.
  In other words, it is a subproblem with exactly the same structure as the original one, except removing the decision variables for the vertices of a single clique.
  
  Finally, we look at families of cutting planes that we can derive.
  By adding up the three constraints corresponding to the edges of any triangle $\{u,v,w\}$, we obtain the implication $2(x_u + x_v + x_w) \le 3$.
  Since all variables are integer-restricted, we can infer that $x_u + x_v + x_w \le \floor{3/2} = 1$ for every clique.
  Each such cut corresponds to a change of $\cutbd=1/2$ in the objective,
  and there exists one such cut for every clique of three vertices.
  
  Hence, by \Cref{thm:SVBcutGood_lequalsr}, we have that, not counting cut nodes, the optimal depth of the $\svbc[\lrc]$ tree that proves the bound $\targetbound = m/2$ is $\delta^* = \floor{\log_2\ceil{\rbd/\cutbd}} = 0$. This implies that the optimal number of cut rounds is 
    \[
      k^* 
      = \ceil{\frac{\targetbound}{\cutbd}} 
      = \ceil{ \frac{m/2}{1/2} } 
      = m,
     \]
  for a corresponding tree with $m$ total nodes, compared to a pure branching tree
  that would %
  have depth $\ceil{\targetbound/\rbd} = m$ and thus $2^{m+1} - 1$ nodes, 
  which is exponentially many more than if cuts are used. \exqed
\end{example}

\section{Diminishing Cut Strength}
\label{sec:svbhc}

In the \SVBC model, the assumption of constant cut strength $\cbd$ implies that a tree with only cut nodes proving a bound $\targetbound$ has size $1+\ceil{\targetbound/\cbd}$, growing linearly with $\targetbound$.
Meanwhile, the tree size to prove the same bound by only branch nodes is exponential in $\targetbound$.
While \cref{ex:independent-set} illustrates that there exist cases where the constant cut strength assumption is satisfied,
a more realistic setting would reflect the empirically-observed phenomenon of diminishing marginal bound improvements from cuts~\citep{balas2010enumerative,dey2022cutting}.
In this section, we study tree size ($\HCf = \allones$) when cuts deteriorate in strength across rounds, for the special case that $\lbd = \rbd$.

\subsection{Empirical Motivation for Worsening Cuts}

\begin{figure}
    \centering
    \begin{subfigure}[b]{.7\textwidth}
        \includegraphics[width=\linewidth]{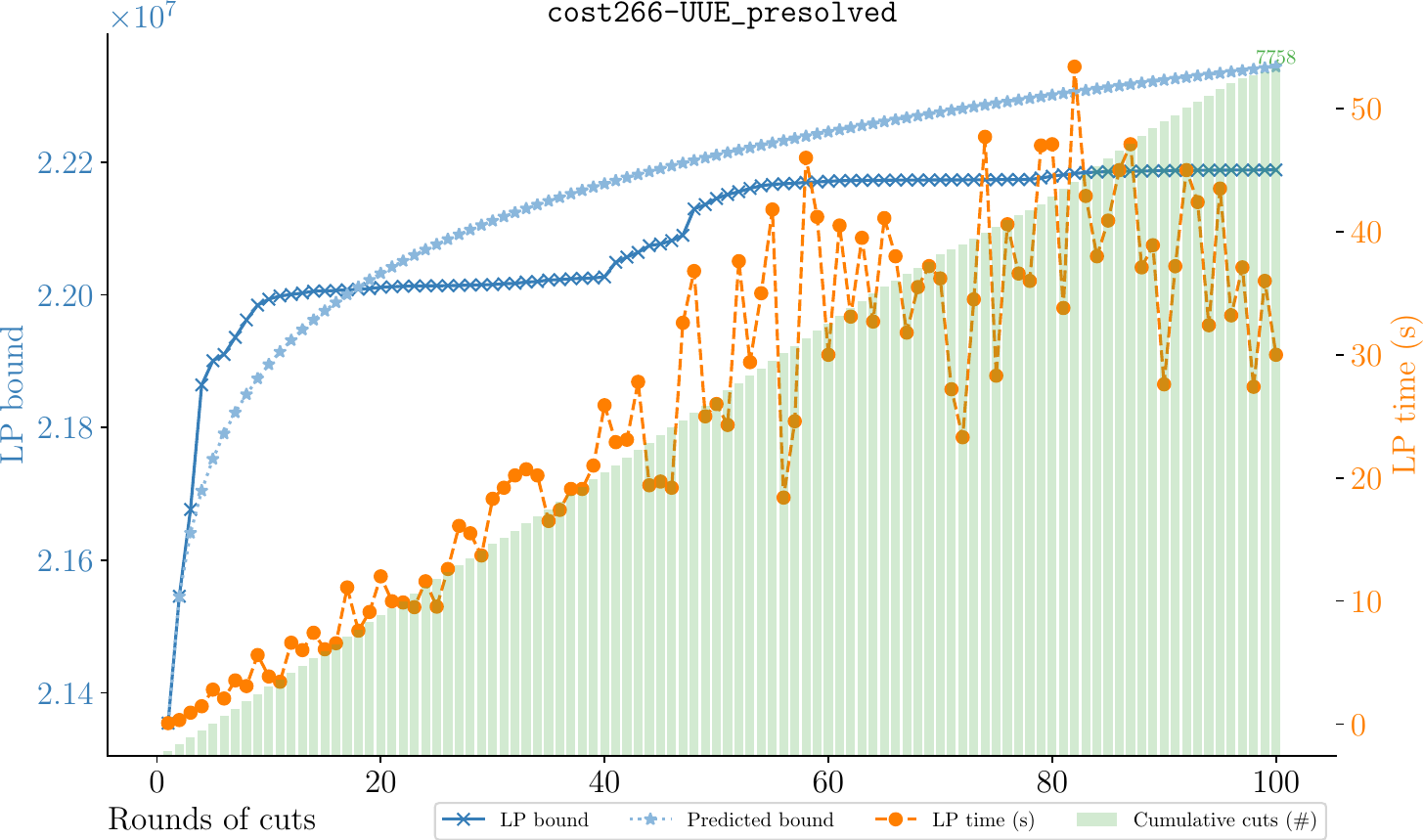}
    \end{subfigure}
    \\\vspace{1em}
    \begin{subfigure}[b]{.7\textwidth}
        \includegraphics[width=\linewidth]{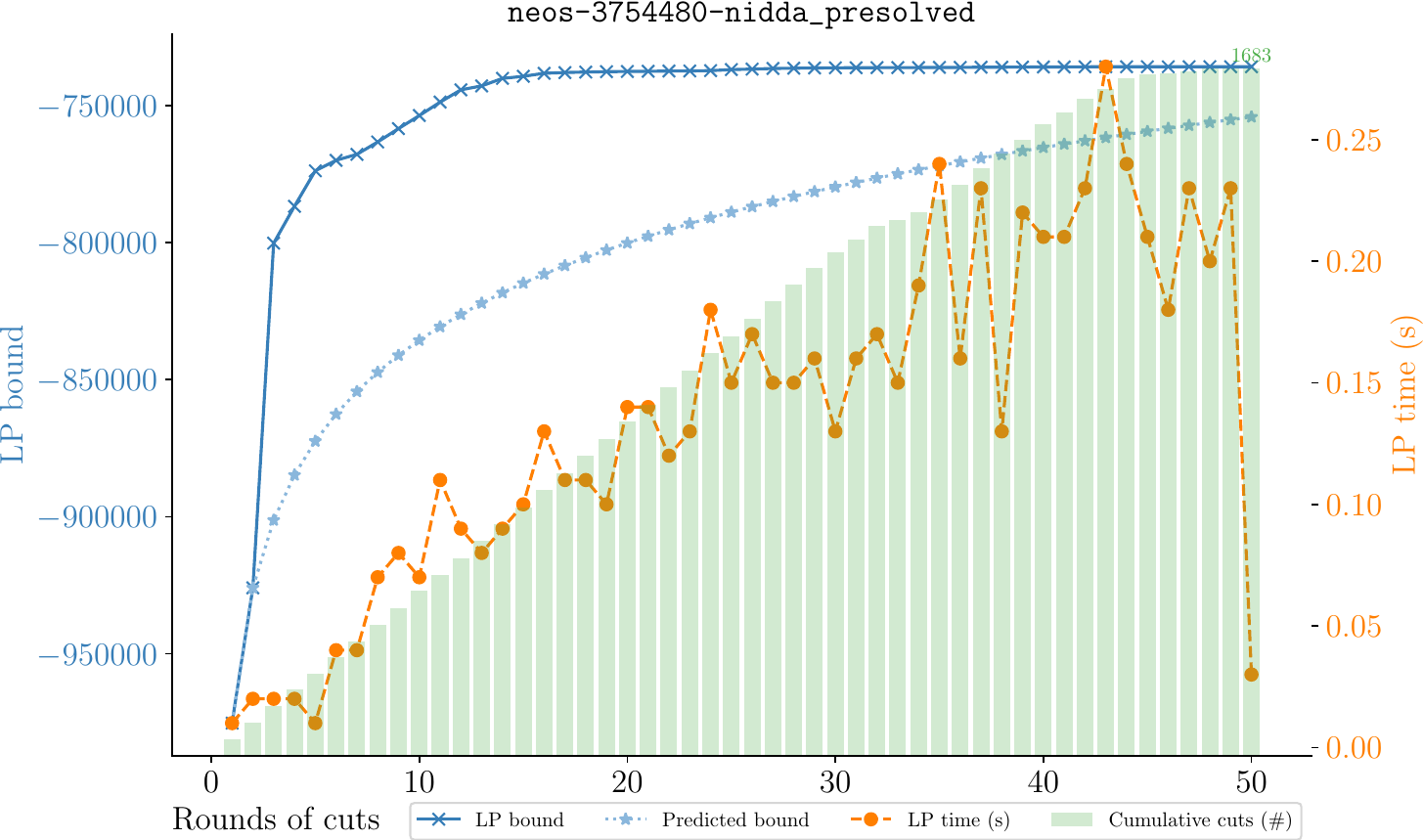}
    \end{subfigure}
    \caption{Applying rounds of Gomory cuts on two MIPLIB 2017 instances (after preprocessing) yields a diminishing bound improvement (LP bound) and a linear tendency for LP resolve time (LP time). The overlayed bar plot for each instance shows the cumulative number of cuts added after each round. The predicted bound using the improvement from the first round follows a logarithmic function that is similar to the actual bound evolution.}
    \label{fig:two-instances}
\end{figure}

We first test 12 instances from the 2017 Mixed Integer Programming Library~\cite{MIPLIB2017} as examples of the evolution of the bound as rounds of cuts are applied at the root node on an optimization solver.
The subset of instances is based on auxiliary testing showing that they have linear relaxations that solve relatively quickly and that Gomory mixed-integer cuts~\cite{Gomory60_gmic} nontrivially affect the bound.
We do not claim these instances are representative.

Specifically, we apply rounds of Gomory mixed-integer cuts to instances that are first presolved using Gurobi~\cite{Gurobi}.
Cuts are computed in each round with the CglGMI implementation in the Cut Generation Library~\cite{CglNoVersion}.
The linear relaxation after each round of cuts is solved using Clp~\cite{Clp}.
Cut generation is terminated after either 
    100 rounds of cuts have been applied, 
    one hour has elapsed, 
    or no cuts are generated in a given round.
Experiments are performed with a single thread on HiPerGator, a shared cluster through Research Computing at the University of Florida.

\cref{fig:two-instances} depicts the results for two instances.
Each plot shows four time series after rounds of cuts:
    the linear relaxation optimal value (``LP bound''), 
    the ``Predicted bound'' corresponding to the model in \Cref{sec:svbhc-def},
    the number of seconds it takes to solve the linear relaxation (``LP time''), 
    and the total number of cuts in the relaxation (``Cumulative cuts'').
The results for the remaining 10 instances are plotted in \cref{fig:more-instances} in \Cref{appendix:more-instances}.
For the predicted bound series, the bound in round $t$ is based on the first round of cuts: if $z_0$ is the initial linear programming relaxation optimal value, and $z_1$ is the value of the linear relaxation after one round of cuts, then the predicted bound $\tilde{z}_t$ at round $t$ is computed as $\tilde{z}_t \defeq \tilde{z}_{t-1} + (z_1-z_0)/t$, where $\tilde{z}_0 \defeq z_0$.

Across the instances, we observe that cuts tend to provide diminishing marginal bound improvements as more rounds are applied.
Furthermore, the predicted bound follows a logarithmic function that reflects the general trend in the bound, though it becomes less accurate in later rounds when there is more substantial tailing off in bound improvement.

The model introduced next only aims to capture the relative decrease in how cuts affect the bound across rounds.
The time to solve the linear relaxation as a function of number of cut rounds is further discussed in \Cref{sec:general-cut-time}.

\subsection{\SVBWCText}
\label{sec:svbhc-def}
We define a model in which the total bound improvement by $k$ cuts is the $k$th harmonic number scaled by a constant parameter $\cbd$,
so that the number of cuts needed to prove a bound $\targetbound$ grows exponentially in $\targetbound$.
Thus, we have two exponential-time procedures (pure cutting and pure branching) that can work together to prove the target bound.
Let $H:\nonnegints \to \nonnegreals$ denote the $k$th harmonic number $H(k) \defeq \sum_{i=1}^k 1/i$.

\begin{definition}[\SVBWCText{} (\SVBWC)]
\label{def:svbwc}
A \bc tree is a \emph{\SVBWCText{} (\SVBWC)} tree with parameters $\lrcw[\lbd,\rbd;\cutbd,\HCf]$, or $\svbwc$ tree, if the bound improvement value associated with each branch node is $(\lbd, \rbd)$, the bound improvement by cut node $k$ is $\cutbd / k$, and the \timefunction{} is $\HCf$.
\end{definition}

\Cref{lem:cutFirst} can be extended to this setting when $\HCf = \allones$.
Hence, without loss of generality, we only need to consider cut-and-branch trees.

When the bound improvement by each cut node is a constant $\cbd$, \cref{thm:SVBcutGood_lequalsr} shows that for any target $\targetbound$,
at most $\rbd \depth^*$ of the bound (a constant independent of $\targetbound$) is proved by branching, and the rest by cutting.
However, this is no longer true when cuts exhibit diminishing returns.
The proof of \Cref{thm:SVBcutGood_lequalsr} hinges on \Cref{lem:svbc-opt-to-cut-layers},
from which we know that analyzing the optimal number of root-node cuts is equivalent to understanding the optimal depth of the branching component. %
As \Cref{lem:minPossibility} will show, it continues to be sufficient to analyze the number of branching layers in the \SVBWC setting; 
the main difference is that we no longer have an exact analytical expression for
the number of cuts such that the branching component has depth $\depth$, 
which requires us to find the minimum integer $k$ such that cuts prove a bound of $\targetbound - \depth \rbd$, i.e.,
	\[
		 \targetbound - \depth \rbd
		 \le
		\sum_{i=1}^k \frac{\cbd}{i}
		= \cbd H(k).
	\]
Define
    \begin{equation*}
        \ncutsvbwc(\depth) \defeq 
            \min_k \left\{
            k \in \nonnegints \suchthat 
                H(k) \ge \frac{\targetbound - \depth \rbd}{\cbd}
            \right\}.
    \end{equation*}
Note the similarity to the definition of $\ncutsZ(\depth) = \max\{0, \ceil{(\targetbound - \depth \rbd) / \cbd}\}$.
As there is currently no proved exact analytical expression for $H(k)$ and $\ncutsvbwc(\depth)$, we avail of well-known bounds on these functions
to \emph{approximate} the value $k$ for the minimum number of cuts needed to achieve a branching depth of $\depth$.

Let $H^{-1}(x) \defeq \min_k \{ k \in \nonnegints \suchthat H(k) \ge x \}$,
so that 
    $\ncutsvbwc(\depth) = H^{-1}((\targetbound-\depth \rbd)/\cbd)$.
\cref{lem:harmIneq} restates well-known bounds on $H(k)$ and $H^{-1}(x)$.
\begin{lemma} \label{lem:harmIneq}
\mbox{}
\begin{enumerate}
    \item For any $z\in \posints$, $\ln (z+1) < H(z) \le \ln(z) + 1$.
    \item It holds that $H^{-1}(0) = 0$, $H^{-1}(x) = 1$ for any $x \in (0,1]$, and for any $x\in (1, \infty)$, $e^{x-1} \le H^{-1}(x) < e^{x}-1$.
\end{enumerate}
\end{lemma}

\subsection{Overview of \cref{alg:SVBWC} Approximating Optimal Number of Cuts}
We consider the case where $\lbd = \rbd$, but could be different from $c$.
In \Cref{alg:SVBWC}, we approximate the number of cut nodes in a minimal \SVBWC tree.
Our main result, stated in \Cref{thm:ApxMAIN}, is that the tree with this number of cuts at the root and the remaining bound proved by branching is no more than a multiplicative factor larger than the minimal-sized tree.

\begin{algorithm}[b]
\begin{algorithmic}[1]
\Require $\rbd, \cbd, \targetbound$.
\Ensure Number of cuts $k$ to be used before proving the remaining bound by branching.
\State 
    $\ubFunOptCont \gets 
      \left({
            \targetbound
            + \cbd \ln (\frac{\rbd}{\cbd \ln 4})
            }\right)/\left({
            \rbd + \cbd \ln 2
            }\right).
    $
    \Comment{Continuous minimizer of $\ubFun$ in \Cref{lem:ubMin}.}
\State
	$\hat{\depth}^* \gets \floor{(\targetbound - \cbd) / \rbd}$.
	\Comment{Maximum depth for which \Cref{lem:TreeSizeBound} bounds apply.}
\State \label{step:depth12_round-continuous-minimizer}%
    $\depth_1 \gets \lfloor \ubFunOptCont \rfloor$,
    $\depth_2 \gets \lceil \ubFunOptCont \rceil$.
    \Comment{Integer minimizer for $\ubFun$ is a rounding of $\ubFunOptCont$.}
\State \label{step:depth3_largest-possible}%
    $\depth_3 \gets \hat{\depth}^* + 1$.
    \Comment{Only other possible minimal tree branching depth per \Cref{lem:minPossibility}.}
\State \label{step:return-ncuts}%
    Return $\ncutsvbwc(\delta)$ for $\delta \in \argmin_\depth \left\{ \ncutsvbwc(\depth) + 2^{\depth + 1} - 1 \suchthat \depth \in \{\depth_1,\depth_2,\depth_3\} \right\}$.
\end{algorithmic}
\caption{Approximating the number of cuts to be used in \SVBWC}\label{alg:SVBWC}
\end{algorithm}

\begin{theorem}\label{thm:ApxMAIN}
When $\lbd = \rbd$, let $T$ denote the $\svbwc[\lrc]$ cut-and-branch tree $T$ that proves a bound of $\targetbound$ using the number of cut nodes prescribed by \cref{alg:SVBWC}.
Let $T^\star$ denote a minimal-size \SVBWC tree proving bound $\targetbound$.
Then $\treetime(T) \le \max \{8,  e^{1+{\rbd}/{\cbd}} \} \treetime(T^\star)$.
\end{theorem}

We recommend deferring the reading of \cref{alg:SVBWC} until the end of the section, as its meaning is rooted in the results that follow.
Intuitively, the algorithm is analogous to \Cref{thm:SVBcutGood_lequalsr}, in that an approximately-optimal tree size can be obtained from checking only one of a few possible values for the branching component depth.
We must compute $H^{-1}$ for one of these values, but this is inexpensive given the conjectured tight bounds mentioned above.

The rest of the section is dedicated to proving \cref{thm:ApxMAIN} by a series of lemmas, organized as follows. 
\cref{lem:minPossibility} significantly reduces the search space of the optimal number of cuts to finitely many options, based on the depth of the branching component of the tree.
\cref{lem:TreeSizeBound} provides bounds on tree size as a function of the branching component depth.
These bounds apply at all but the largest possible depth from \cref{lem:minPossibility}.
\cref{lem:lbMin,lem:ubMin} find the continuous minimizers of the lower- and upper-bounding functions of the total tree size.
As the depth of the branching component must be integral, convexity implies that the integer minimizers of the bounding functions can be obtained by rounding the continuous minimizers.
\cref{lem:lbubMinDiff} bounds the difference between the integer minimizers of the lower and upper bound functions. 
Finally, the proof of \cref{thm:ApxMAIN} shows that a branching component depth set as the integer minimizer of the upper-bounding function provides the desired approximation factor to minimal tree size, when the upper-bounding function applies, and the only other possible depth is explicitly checked.

\subsection{Bounding \SVBWC{} Tree Sizes}
\Cref{lem:minPossibility} is a refined analogue of \Cref{lem:svbc-opt-to-cut-layers}, stating that the optimal number of cut nodes in a minimal-sized \SVBWC tree must correspond to $\ncutsvbwc(\depth)$ for a restricted possible range of $\depth$, given that we allow for $\cbd > \rbd$ in this context.
This restricted range of $\depth$ is later used to apply bounds on tree size in \Cref{lem:TreeSizeBound}.
Recall that the pure branching tree has depth
    $
        \depthmax \defeq \ceil{\targetbound / \rbd}.
    $

\begin{lemma}\label{lem:minPossibility}
In any minimal $\svbwc[\lrc]$ tree with $\lbd=\rbd$ that proves bound 
    $\targetbound$, 
the number of cut nodes in the tree is 
$\ncutsvbwc(\depth)$, for some branching depth
    $\depth \in \{0,\ldots,\max\{0,\depthmax - \floor{\cbd/\rbd}\}\}.$
\end{lemma}
\begin{proof}
    As in \Cref{lem:svbc-opt-to-cut-layers}, it is clear that the optimal number of cut nodes is $\ncutsvbwc(\depth)$ for some depth $\depth$ of the branching component.
    If a single cut node proves at least as much bound as two branch nodes,
    it is better to add the cut rather than branch, as long as the target bound has not been attained.
    Hence, the minimal-sized \SVBWC tree will have at least $k$ cuts, where $k$ is the maximum integer such that 
        $ \cbd / k \ge \rbd $
    or $H(k) \ge \targetbound$.
    If the latter holds, then the optimal branching depth is $0$.
    Otherwise, for a large enough target bound,
    the former inequality implies that at least $\floor{\cbd / \rbd}$ cut nodes will be used.
    Moreover, as each of these cut nodes will yield at least $\rbd$ bound improvement, the remaining bound by branching only requires a depth of at most $\depthmax - \floor{\cbd/\rbd}$.
\end{proof}

When $\lbd = \rbd$, an $\svbwc[\lrc]$ cut-and-branch tree that proves bound $\targetbound$ and for which the depth of the branching component is $\depth$
has size $\ncutsvbwc(\depth) + 2^{\depth+1} - 1$,
but $\ncutsvbwc(\depth)$ is not explicit.
In the lemma below, we provide functions $\lbFun$, $\ubFun$ which respectively provide lower and upper bounds for \SVBWC tree sizes.

\begin{lemma} \label{lem:TreeSizeBound}
    When $\lbd = \rbd$, consider an $\svbwc[\lrc]$ cut-and-branch tree proving a target bound $\targetbound$,
    where the branching component has depth $\depth$.
    Let $\targetbound_\depth \defeq \targetbound-\depth \rbd$ denote the bound to be proved by cut nodes when the branching component has depth $\depth$.
    Then the size of the tree is
    equal to $2^{\depth+1}-1$ if $\targetbound_\depth \le 0$,
    equal to $2^{\depth+1}$ if $\targetbound_\depth \in (0, \cbd]$,
    and otherwise, for all
        $
			\depth 
			\le 
			\hat{\depth}^* \defeq \floor{(\targetbound - \cbd) / \rbd},
		$
    \begin{enumerate}
        \item
            at least 
            $
                \lbFun(\depth) \defeq 
                    e^{\targetbound_\depth / \cbd - 1}
                    + 2^{\depth+1} - 1.
            $
        \item at most 
            $
                \ubFun(\depth) \defeq 
                e^{\targetbound_\depth / \cbd}
                + 2^{\depth+1} - 2.
            $
    \end{enumerate}
\end{lemma}
\begin{proof}
    Using $\ncutsvbwc(\depth) = H^{-1}(\targetbound_\depth/\cbd)$,
    we apply \Cref{lem:harmIneq} to the size of a tree of branching component depth $\depth$.
    Specifically,
    $\ncutsvbwc(\depth) = 0$ if $\targetbound_\depth \le 0$,
    $\ncutsvbwc(\depth) = 1$ if $\targetbound_\depth \in (0,\cbd]$,
    and otherwise
        $
            e^{\targetbound_\depth / \cbd - 1} 
            \le \ncutsvbwc(\depth) 
            < e^{\targetbound_\depth / \cbd} - 1.
        $
\end{proof}

Next, \cref{lem:lbMin,lem:ubMin} identify the minimizers of the lower and upper bounding functions identified in \cref{lem:TreeSizeBound}, with no integrality restrictions on the depth $\depth$. We will then argue in \cref{lem:lbubMinDiff} that since this is a one-dimensional convex minimization problem, the optimum after imposing integrality restrictions on $\depth$ is a rounding of the continuous optimum.

\begin{lemma} \label{lem:lbMin}
The unique (continuous) minimum of $\lbFun(\depth)$ defined in \cref{lem:TreeSizeBound} occurs at
    \[ 
       \lbFunOptCont \defeq 
        \frac{
            \targetbound
            + \cbd \left(
                \ln (\frac{\rbd}{\cbd \ln 4}) - 1
            \right)
            }{
            \rbd + \cbd \ln 2
            }.
    \]
\end{lemma}
\begin{proof}
    The function $\lbFun(\depth)$ is a sum of two strictly convex differentiable functions. 
    Thus, $\lbFun(\depth)$ is also a strictly convex differentiable function.
    The derivative with respect to $\depth$ is
    \begin{align*}
        \lbFun'(\depth) 
            =
            -\frac{\rbd}{\cbd} e^{\frac{\targetbound - \depth \rbd}{\cbd}-1}
            + 2^{\depth+1} \ln 2.
    \end{align*}
    Setting the above to zero, the unique minimum of $\lbFun(\depth)$ is at $\lbFunOptCont$.
\end{proof}

\begin{lemma} \label{lem:ubMin}
The unique (continuous) minimum of $\ubFun(\depth)$ defined in \cref{lem:TreeSizeBound} occurs at
    \[
       \ubFunOptCont \defeq 
       \frac{
            \targetbound
            + \cbd \ln (\frac{\rbd}{\cbd \ln 4})
        }{
            \rbd + \cbd \ln 2
        }.
    \]
\end{lemma} 
\begin{proof}
    The function $\ubFun(\depth)$ is a sum of two strictly convex differentiable functions. 
    Thus, $\ubFun(\depth)$ is also a strictly convex differentiable function.
    The derivative with respect to $\depth$ is
    \begin{align*}
        \ubFun'(\depth) 
            =
            -\frac{\rbd}{\cbd} e^{\frac{\targetbound - \depth \rbd}{\cbd}}
            + 2^{\depth+1} \ln 2.
    \end{align*}
    Setting the above to zero, the unique minimum of $\ubFun(\depth)$ is at $\ubFunOptCont$.
\end{proof}

Having found the continuous minima of $\lbFun$ and $\ubFun$, now we prove that the integer minimizers of $\lbFun$ and $\ubFun$ cannot be too far away from each other.

\begin{lemma} \label{lem:lbubMinDiff}
    Let $\lbFunOpt$ and $\ubFunOpt$ 
    be minimizers of $\lbFun(\depth)$ and $\ubFun(\depth)$, as defined in \cref{lem:TreeSizeBound}, over the set of nonnegative integers.
    Then, $-1 \le \ubFunOpt - \lbFunOpt \le 2$. 
\end{lemma}
\begin{proof}
    The integer minimizer of a one-dimensional convex function is either the floor or ceiling of the corresponding continuous minimizer.
    Hence, 
        $
        \lbFunOpt \in \{\lfloor \lbFunOptCont \rfloor, \lceil \lbFunOptCont \rceil\}
        $
    and
        $
        \ubFunOpt \in \{\lfloor \ubFunOptCont \rfloor, \lceil \ubFunOptCont \rceil\}.
        $
    Let $\epsilon \defeq \ubFunOptCont - \lbFunOptCont$.
    Using $\ln 2 > 2/3$, we bound
        \[
            \epsilon
            = \frac{1}{\rbd/\cbd + \ln 2}
            \in (0, 1.5).
        \]
    We then have that
        \[
            \ubFunOpt - \lbFunOpt
            \le
            \lceil{\ubFunOptCont}\rceil
            - \lfloor{\lbFunOptCont}\rfloor
            =
            \lceil{\lbFunOptCont+\epsilon}\rceil
            - \lfloor{\lbFunOptCont}\rfloor
            \le
            \lceil \epsilon \rceil
            \le 2.
        \]
    Similarly,
         \[
            \ubFunOpt - \lbFunOpt
            \ge
            \lfloor{\ubFunOptCont}\rfloor
            - \lceil{\lbFunOptCont}\rceil
            =
            \lfloor{\lbFunOptCont+\epsilon}\rfloor
            - \lceil{\lbFunOptCont}\rceil
            \ge
            \lfloor{\lbFunOptCont}\rfloor
            - \lceil{\lbFunOptCont}\rceil
            \ge -1.
            \qedhere
        \]
\end{proof}

From \cref{lem:lbubMinDiff}, there is a possibility for $\lbFunOpt$ to be equal to, less than, or greater than $\ubFunOpt$, leading to prescribing different numbers of cut nodes from both bounds.
To complete the proof, we show that the tree size when using the number of cuts prescribed by the lower bound is not too different from the tree size when using cuts as prescribed by the upper bound.
This, in turn, implies the desired approximation with respect to the minimal tree size, when combined with checking the additional possible branching component depth at which the approximations in \Cref{lem:TreeSizeBound} do not apply.

\subsection{Proof of \cref{thm:ApxMAIN}}
\begin{proof}[Proof of \cref{thm:ApxMAIN}]
    Let $T_\depth$ denote the $\svbwc[\lrc]$ tree that proves bound $\targetbound$ with branching component having depth $\depth$ and $\ncutsvbwc(\depth)$ cut nodes at the root.
    Let $\ubFunOptCont$ be the continuous minimizer of the upper-bounding function on tree size $\ubFun(\depth)$ from \Cref{lem:ubMin},
    and let $\ubFunOpt \in \{ \lfloor \ubFunOptCont \rfloor, \lceil \ubFunOptCont \rceil \}$ be the integer minimizer, as defined in \cref{lem:lbubMinDiff}.
    From \Cref{lem:TreeSizeBound}, the bounds on $\lbFun$ and $\ubFun$ only apply for 
    	$
			\depth 
			\le 
			\hat{\depth}^* \defeq \floor{(\targetbound - \cbd) / \rbd}.
		$
    At the same time, from \Cref{lem:minPossibility}, the maximum possible optimal branching depth is 
        $
            \ceil{\targetbound / \rbd} - \floor{\cbd/\rbd} 
            \le 
            \hat{\depth}^* + 1$. 
    In \Cref{alg:SVBWC}, we explicitly check $\treesize(T_{\hat{\depth}^*+1})$,
    and for the remaining possibilities, we will prove that it suffices to check $\treetime(T_{\ubFunOpt})$ to get the desired approximation of the size of a minimal \SVBWC tree $T^\star$ that proves bound $\targetbound$.
    
    Note that if $\ubFunOpt > \hat{\depth}^*$,
    then the bounds on $\ubFun$ and $\lbFun$ do not apply, but for this case, we do not need to rely on the below approximation, as we are explicitly checking the tree size for depth $\hat{\depth}^* + 1$,
    the only other possible branching depth in a minimal tree.
    Thus, for the ensuing discussion,
    assume that $\ubFunOpt \le \hat{\depth}^*$.

    Let $\lbFunOptCont$ be the continuous minimizer of $\lbFun(\depth)$ from \Cref{lem:lbMin},
     and let $\lbFunOpt$ be the integer minimizer from \Cref{lem:lbubMinDiff}.
    Let $\depth^\star$ be the branching component depth of $T^\star$.
    Then
        \[
            \lbFun(\lbFunOpt)
            \le
            \lbFun(\depth^\star)
            \le
            \treetime(T^\star)
            \le
            \treetime(T_{\ubFunOpt})
            \le
            \ubFun(\ubFunOpt).
        \]
    We have that $\lbFun(\lbFunOpt) \le \lbFun(\depth^\star)$
    because $\lbFunOpt$ is an integer minimizer of $\lbFun$.
    The second inequality holds because $\lbFun(\depth)$ is a lower bound on the size of a tree with branching depth $\depth$.
    The next inequality follows from the minimality of $T^\star$.
    Finally, $\treetime(T_{\ubFunOpt}) \le \ubFun(\ubFunOpt)$ as $\ubFun(\depth)$ upper bounds the size of a tree having branching depth $\depth$.
    Thus, we have that
        \[
            1 
            \le 
            \frac{\treetime(T_{\ubFunOpt})}{\treetime(T^\star)}
            \le
            \frac{\ubFun(\ubFunOpt)}{\lbFun(\lbFunOpt)}.
        \]
    The goal is to bound
        $
        {\treetime(T_{\ubFunOpt})} / {\treetime(T^\star)},
        $
    which we pursue by first bounding
        \begin{equation}
        \label{eq:ApxBnd}
            \frac{\ubFun(\ubFunOpt)}{\lbFun(\lbFunOpt)}
            =
            \frac{e^{(\targetbound - \ubFunOpt \rbd) / \cbd} + 2^{\ubFunOpt + 1} - 2}{e^{(\targetbound - \lbFunOpt \rbd) / \cbd - 1} + 2^{\lbFunOpt + 1} - 1}.
        \end{equation}
    
We bound this ratio for three cases based on the values of $\lbFunOpt$ and $\ubFunOpt$. 
 \begin{enumerate}[label={\textit{Case \arabic*.}},ref={Case~\arabic*},leftmargin=*,topsep=2pt, partopsep=2pt]
 \item $\lbFunOpt > \ubFunOpt$. 
    In this case, $\ncutsvbwc(\ubFunOpt) > \ncutsvbwc(\lbFunOpt)$,
    i.e., our tree $T_{\ubFunOpt}$ has more cuts than $T_{\lbFunOpt}$.
    We upper bound the number of ``extra'' cuts this might result in.
    It holds that
        $
            2^{\ubFunOpt + 1}
            \le 2^{\lbFunOpt}
            < 2^{\lbFunOpt + 1} +1.
        $
    Thus, using \cref{eq:ApxBnd} and that
    the exponent of $e$ in the numerator is larger than in the denominator,
        \begin{equation*}
            \frac{\ubFun(\ubFunOpt)}{\lbFun(\lbFunOpt)}
            \le
            \frac{e^{(\targetbound - \ubFunOpt \rbd) / \cbd} + 2^{\lbFunOpt + 1} - 1}{e^{(\targetbound - \lbFunOpt \rbd) / \cbd - 1} + 2^{\lbFunOpt + 1} - 1}
            \le 
            \frac{e^{(\targetbound - \ubFunOpt \rbd) / \cbd} }{e^{(\targetbound - \lbFunOpt \rbd) / \cbd - 1}}
            = 
            e^{1 + \frac{\rbd(\lbFunOpt-\ubFunOpt)}{\cbd}}
            \le 
            e^{1 + \frac{\rbd}{\cbd}},
        \end{equation*}
        where the final relation follows from $\lbFunOpt - \ubFunOpt \le 1$ from \cref{lem:lbubMinDiff}.
  
    \item $\lbFunOpt = \ubFunOpt$.
        \begin{equation*}
            \frac{\ubFun(\ubFunOpt)}{\lbFun(\ubFunOpt)}
            \le 
            \frac{e^{(\targetbound - \ubFunOpt \rbd) / \cbd} - 1}{e^{(\targetbound - \ubFunOpt \rbd) / \cbd - 1}}
            < e. 
        \end{equation*}
  
    \item %
        $\lbFunOpt < \ubFunOpt$.
        In this case, $\ncutsvbwc(\ubFunOpt) < \ncutsvbwc(\lbFunOpt)$.
        The tree we are evaluating, $T_{\ubFunOpt}$,
        has fewer cuts than as suggested by the lower bounding function. 
        We have to ensure that cutting less has not made the branching part of the tree too large.
        Hence, applying $\ubFunOpt-\lbFunOpt \le 2$
        from \cref{lem:lbubMinDiff},
        and $2^{\lbFunOpt+1} \ge 2$,
        \begin{equation*}
            \frac{\ubFun(\ubFunOpt)}{\lbFun(\ubFunOpt)}
            \le  
              \frac
              {2^{\ubFunOpt+1}-1}
              {2^{\lbFunOpt+1}-1}
            =
             \frac
             {2^{\ubFunOpt-\lbFunOpt}-1/2^{\lbFunOpt+1}}
              {1-1/2^{\lbFunOpt+1}}
            \le
             \frac
             {2^2-1/2^{\lbFunOpt+1}}
             {1/2}
            \le
              8 - 2^{-\lbFunOpt}
           \le 8.
        \end{equation*}
    \end{enumerate}
    Combining the bounds above gives the result.
\end{proof}

\subsection{Cuts Prove a Constant Portion of the Bound}
\cref{thm:ApxMAIN} proves an approximation to the optimal number of cut nodes in an optimal $\SVBWC$ tree.
Next, as a complement, \cref{thm:apxGrowth} states that, in the limit, the number of cut nodes prescribed by \cref{alg:SVBWC} proves a \emph{constant} fraction of the target bound relative to the portion proved by branching.

\begin{theorem}\label{thm:apxGrowth}
Let $\lbd=\rbd$ and $\cbd > 0$ be fixed.
For a given target bound, consider a $\svbwc[\lrc]$ cut-and-branch tree
where the number of nodes is calculated via \cref{alg:SVBWC}.
As the target bound $\targetbound$ goes to infinity,
the fraction of the bound proved by the cut nodes converges to the constant
${\cbd \ln 2}/({\rbd + \cbd \ln 2})$. 
\end{theorem}
\begin{proof}
    \Cref{alg:SVBWC} evaluates three different branching depths to determine the number of cut nodes that approximately minimize overall tree size.
    
    First, consider a tree with branching depth $\depth_3 = \ceil{\targetbound / \rbd} - \floor{\cbd/\rbd}$ from \cref{step:depth3_largest-possible} of the algorithm,
    which is an upper bound on the depth when $k$ cuts are added, where $k$ is the maximum integer such that $\cbd / k \ge \rbd$.
    As $\targetbound$ increases, nearly all of the bound is proved by branching in this case, with the bottom layer of the tree containing $2^{\depth_3}$ leaf nodes.
    There exists a $k'$ such that adding $k' - k$ more cut nodes will satisfy $\cbd/(k+1)+\cdots+\cbd/k' \ge \rbd$, where $k'$ is independent of $\targetbound$.
    Hence, for sufficiently large $\targetbound$, $\depth_3$ will not be the minimizer selected in \cref{step:return-ncuts}.

    Next, recall from \Cref{lem:ubMin} that the continuous minimizer of $\ubFun$,
    which provides an upper bound on the total size of the tree as a function of the depth of the branching component,
    is $\ubFunOptCont$,
    and the other two possible outputs are
    $\depth_1 = \lfloor{\ubFunOptCont}\rfloor$ and 
    $\depth_2 = \lceil{\ubFunOptCont}\rceil$
    from
    \cref{step:depth12_round-continuous-minimizer}.

    Since $\delta_1$ and $\delta_2$ are at most one unit away from $\ubFunOptCont$,
    the amount of bound proved by branching is in the range
        $[(\ubFunOptCont-1)\rbd, (\ubFunOptCont+1)\rbd]$.
    Substituting in the value of $\ubFunOptCont$ and dividing by $\targetbound$,
        \[
            \frac{\ubFunOptCont \rbd}{\targetbound}
            \pm \frac{\rbd}{\targetbound}
            =
            \frac{
                \targetbound \rbd
                + \rbd \cbd \ln (\frac{\rbd}{\cbd \ln 4})
            }{
                \targetbound (\rbd + \cbd \ln 2)
            }
            \pm \frac{\rbd}{\targetbound},
        \]
    both the upper and lower bound of the fraction of bound proved by branching nodes 
    tends to 
        ${\rbd}/({\rbd + \cbd \ln 2})$
    as $\targetbound$ goes to infinity, 
    implying that the fraction of bound proved by branching nodes also tends to the same value.
    This further implies that as $\targetbound \to \infty$,
    the fraction of bound proved by cut nodes is 
        \begin{equation*}
            1 - \frac{\rbd}{\rbd + \cbd \ln 2}
            =
            \frac{\cbd \ln 2}{\rbd + \cbd\ln 2}. \qedhere
        \end{equation*}
\end{proof}

\cref{thm:apxGrowth} provides an indication of the tradeoff between cutting and branching in the harmonically-worsening cuts model, and it applies, for example, to an increasingly difficult family of instances (quantified by an increasing target bound), for a fixed relative strength of cutting and branching.
For example, if the first cut is a factor of $1/\ln 2 \approx 1.44$ stronger than branching, then around half of the bound is proved by cutting, in the limit.
More generally, if $\lambda > 0$ such that $\cbd = \lambda \rbd / \ln 2$,
then approximately
    $
        \lambda / (1+\lambda)
    $
proportion of the target bound is proved by cut nodes as $\targetbound \to \infty$.

\cref{thm:ApxMAIN,thm:apxGrowth} hinge on bounds on harmonic numbers and the function $H^{-1}$.
Improving these bounds can lead to an improvement in the approximation factor, or even an exact algorithm, for the optimal number of cut nodes, and hence of the optimal tree size.
For example, Hickerson~\cite[\href{https://oeis.org/A002387}{A002387}]{oeis} conjectures that, if $n \in \nonnegints$, $H^{-1}(n) = \floor{e^{n-\gamma} + 1/2}$ for $n \ge 2$, where $\gamma$ denotes the Euler-Mascheroni constant, approximately $0.577$.%
\footnote{See references and notes in \url{https://oeis.org/A002387} and \url{https://oeis.org/A004080}.}

\section{Optimizing Tree Time}
\label{sec:general-cut-time}

We now return to the \SVBC setting in which cuts have constant quality.
Whereas the previous sections focus on decreasing the \emph{size} of a \branchandcutx{} tree, in practice the quantity of interest is the \emph{time} it takes to solve an instance.
The two notions do not intersect: it can be that one tree is smaller than another, but because the relaxations at each node solve more slowly in the smaller tree, the smaller tree ultimately solves in more time than the larger one.
This plays prominently into cut selection criteria, as strong cuts can be dense, and adding such cuts to the relaxation slows down the solver.

\subsection{Time-Functions Bounded by a Polynomial}
\label{sec:cut-time-bounded-by-polynomial}

We first show that if the \timefunction{} is bounded above by a polynomial, then for sufficiently large $\targetbound$, it is optimal to use at least one cut node.
\cref{fig:two-instances} provides empirical motivation for this assumption.
The secondary vertical axis is the time (in seconds) to resolve the linear relaxation over (up to) 100 rounds of Gomory cuts.
It can be seen from these plots that, approximately, the time grows \emph{linearly} with the number of added cuts.
Our experiments with additional instances, reported in \Cref{appendix:more-instances}, support the linearity observation, or even a sublinear increase in time, as the number of cuts added in later rounds tends to be smaller compared to the initial rounds.

\begin{theorem}\label{thm:atleastoneHC}
Suppose we have an $\svbc$ tree $T$ and the values of $\HCf$ are bounded above by a polynomial. Then, there exists $\bar{\targetbound} > 0$ such that every $\treeweight$-minimal \SVBC tree proves a bound of $\targetbound >\bar{\targetbound}$ has at least one cut node. 
\end{theorem}
\begin{proof}
Let $\HCf(z) \le 1 + \alpha z^d$ for some $\alpha, d >0$ be the polynomial upper bound for each $z\in \nonnegints$.
A pure branching tree $T_B$ proving a bound $\targetbound$ has at least $2^{ \ceil{{\targetbound}/{\rbd}} +1 }-1$ nodes. The same lower bound holds for $\treeweight(T_B)$.

Now consider a pure cutting tree $T_C$ proving bound $\targetbound$. Such a tree has exactly $k = \ceil{{\targetbound}/{\cutbd}}+1$ nodes.
The tree time for $T_C$ is $\treeweight(T_C) = \sum_{i=0}^{k-1} \HCf(z) \le k + \alpha \sum_{i=1}^{k-1} z^d  \leq k + \alpha(k-1)^{d+1} < p(k)$, where $p$ is some polynomial. %
For sufficiently large values of $\targetbound$, $2^{ \ceil{{\targetbound}/{\rbd}}+1}-1 >   p \left( \ceil{\frac{\targetbound}{\cutbd}}\right) $ for any polynomial $p$, 
implying that a $\treeweight$-minimal tree has at least one cut node. 
\end{proof}

\Cref{thm:atleastoneHC} implies that when cuts affect the time of a tree in a consistent way (through a fixed \timefunction{}) for a family of instances, then cuts are beneficial for a sufficiently hard instance.
A complementary result also holds: if we are considering different cut approaches for a given instance that increasingly slow down node time, then eventually pure branching is optimal.
Specifically,
for any $\lbd$, $\rbd$, $\cutbd$, and $\targetbound$, there exists a \emph{linear} \timefunction{} such that the corresponding $\treeweight$-minimal tree has {\em no} cuts. 
For example, let a pure branching tree of size $\bar{s}$  prove a bound of $\targetbound$.
Then, choosing $\HCf(z) \defeq \bar{s}z+1$ ensures that the pure branching tree is $\treeweight$-minimal.
This is because the tree time of the pure branching tree is $\bar{s}$ while a \bc tree with at least one cut will have a tree time of $\bar{s}+1$.

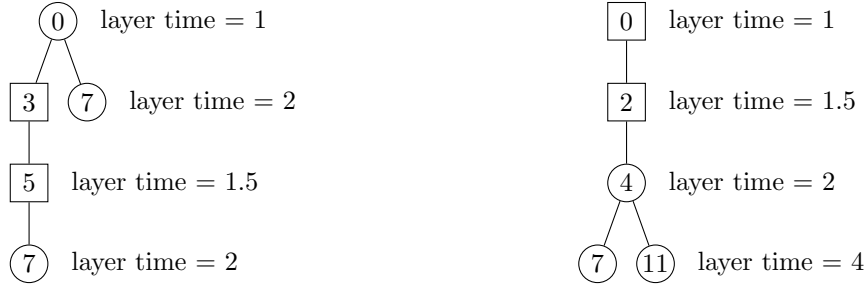
\begin{figure}[t]
    \centering
    \captionsetup[subfigure]{justification=centering}
    \begin{subfigure}{0.45\textwidth}
    \centering
    \begin{forest}
        for tree = {big node}
        [0,lvlwt={1}
            [ 3,big cut [5,big cut,lvlwt={1.5} [7,lvlwt={2}]],
            ],
            [7, lvlwt={2}]
        ]
    \end{forest}
    \end{subfigure}
    \begin{subfigure}{0.45\textwidth}
    \centering
    \begin{forest}
        for tree = {big node}
        [0,big cut,lvlwt={1}
            [ 2,big cut,lvlwt={1.5} [4,lvlwt={2} [7], [11,lvlwt={4}]]
            ]
        ]
    \end{forest}
    \end{subfigure}
    \caption{Consider the \SVBC tree that must prove a bound $\targetbound=7$, with parameters $\lbd = 3$, $\rbd = 7$, $\cutbd = 2$, and $\HCf(z) = z/2 + 1$.
    The \heavycutweight of the first tree is $6.5$
    and of the second tree is $8.5$.
    Thus, cutting at the root node is strictly inferior to cutting at the leaf.
    One can also check that the pure branching tree has a \heavycutweight of $7$ and the pure cutting tree has a \heavycutweight of $10$, showing that the unique $\treeweight$-minimal \bc tree is the tree in the left panel.}
    \label{fig:HeavyCutNoFront}
\end{figure}

Next, we observe that an analogue of \cref{lem:cutFirst} does \emph{not} hold for $\treeweight$-minimality. \Cref{fig:HeavyCutNoFront} provides an example where the unique $\treeweight$-minimal \bc tree has no cuts at the root.
Despite that, for the special case where $\lbd=\rbd$,
we prove in \Cref{thm:root-cuts-suffice} that there is a $\treeweight$-minimal tree having only root cuts.

\begin{theorem}
\label{thm:root-cuts-suffice}
    If $\lbd = \rbd$, then, for any \timefunction{} and target bound $\targetbound$,
    there exists a $\treeweight$-minimal tree with only root cuts.
\end{theorem}

We prove \Cref{thm:root-cuts-suffice} in \Cref{sec:proof-of-thm}.
On the way, we present several intermediate results of independent interest,
which relate properties of general \timefunction{}s to the optimal number and location of cuts in the tree.

\subsection{Minimality of Subtrees and Symmetric Trees}

First, in \Cref{lem:subtrees-are-optimal}, we prove that a subtree of a minimal tree is also minimal.
Given a tree $T$ and any node $u \in T$, let $K_T(u)$ denote the number of cut nodes on the path from the root of $T$ to $u$.

\begin{lemma}
\label{lem:subtrees-are-optimal}
    Let $T$ be a $\treeweight$-minimal $\SVBC{\lrcw}$ tree proving bound $\targetbound$.
    The subtree $T_u$ rooted at $u$ is a $\treeweight$-minimal $\SVBC(\lbd,\rbd;\cbd,\bar{\HCf})$ tree proving bound $\targetbound-\gapfn(u)$,
    where $\bar{\HCf}(z) \defeq \HCf(K_T(u)+z) / \HCf(K_T(u))$ for all $z \in \nonnegints$.
\end{lemma}
\begin{proof}
    Let $T'$ denote any $\SVBC{\lrcw}$ tree proving bound $\targetbound$ that coincides with $T$ for all nodes not in $T_u$.
    Intuitively, if the time for $T'_u$ is less than that of $T_u$, then as both subtrees prove the same bound using the same branch and cut values,
    replacing $T_u$ by $T'_u$ in $T$ would contradict the minimality of $T$.
    
    More directly, the minimality of $T$ implies that $\treeweight_{\HCf}(T) \le \treeweight_{\HCf}(T')$ and hence
         \begin{align*}
             0 &\ge \treeweight_{\HCf}(T) - \treeweight_{\HCf}(T')
             \\
             &= \sum_{v \in T} \HCf(K_T(v)) - \sum_{v \in T'} \HCf(K_{T'}(v))
             \\
             &= \left( \sum_{v \in T \setminus T_u} \HCf(K_T(v)) + \sum_{v \in T_u} \HCf(K_T(v)) \right)
             \\
             &\phantom{=}\ - \left( \sum_{v \in T' \setminus T'_u} \HCf(K_{T'}(v)) + \sum_{v \in T'_u} \HCf(K_{T'}(v)) \right)
             \\
             &= \sum_{v \in T_u} \HCf(K_T(v)) - \sum_{v \in T'_u} \HCf(K_{T'}(v))
             \\
             &= \sum_{v \in T_u} \HCf(K_T(u) + K_{T_u}(v)) - \sum_{v \in T'_u} \HCf(K_{T'}(u) + K_{T'_u}(v))
             \\
             &= \HCf(K_T(u)) \sum_{v \in T_u} \bar{\HCf}(K_{T_u}(v)) - \HCf(K_{T'}(u)) \sum_{v \in T'_u} \bar{\HCf}(K_{T'_u}(v))
             \\
             &= \HCf(K_T(u)) \treeweight_{\bar{\HCf}}(T_u) - \HCf(K_{T'}(u)) \treeweight_{\bar{\HCf}}(T'_u),
         \end{align*}
    which implies that $\treeweight_{\bar{\HCf}}(T_u) \le \treeweight_{\bar{\HCf}}(T'_u)$, as desired.
\end{proof}

Next, in \Cref{lem:l=r_symmetric-tree}, we observe that symmetric trees suffice when $\lbd=\rbd$.

\begin{lemma}
\label{lem:l=r_symmetric-tree}
    If $\lbd=\rbd$, then there exists a $\treeweight$-minimal tree that is symmetric, having the same number of cut nodes along every root-leaf path.
\end{lemma}
\begin{proof}
    The result follows from \Cref{lem:subtrees-are-optimal}, because when $\lbd=\rbd$, if $u$ and $v$ are two nodes at the same depth with $K_T(u) = K_T(v)$,
    then $\gapfn[T](u) = \gapfn[T](v)$.
    Hence if $T$ is $\treeweight$-minimal, then we can assume without loss of generality that the subtree $T_u$ rooted at $u$ is identical to the subtree $T_v$ rooted at $v$.
\end{proof}

\subsection{Adding \texorpdfstring{$k$}{k} Cuts Along Every Root-to-Leaf Path}
\label{sec:k-cuts-along-root-leaf-paths}

We analyze adding $k$ cuts to a generic $\svbc[\lrcw]$ tree
and prescribe how many should be placed before the first branch node. 
{
\begin{figure}[t]
    \centering
    \begin{forest}
        for tree = {small node,l=0.5mm}
        [,cut,nodewt={$\HCf(0)$}
            [,dot node
                [,cut,nodewt={$\HCf(t-1)$}
                    [,nodewt={$\HCf(t)$},s sep=2cm,
                        [,cut,nodewt={$\HCf(t)$}
                          [,dot node
                            [,nodewt={$\HCf(k)$},tikz={ \node [itria,xshift=0pt,fit to=tree,minimum size=.75cm,isosceles triangle apex angle=90,yshift=.1cm,] {$T_L$}; }
                            ]
                          ]
                        ]
                        [,cut,nodewt={$\HCf(t)$}
                            [,dot node
                                [,nodewt={$\HCf(k)$},tikz={ \node [itria,xshift=0pt,fit to=tree,minimum size=.75cm,isosceles triangle apex angle=90,yshift=.1cm,] {$T_R$}; }
                                ]
                            ]
                        ]
                    ]
                ]
            ]
        ]
    \end{forest}
    \caption{What is the optimal choice of the number of cut nodes $\cutind$ to add at the root before we start branching, given a fixed budget of $k$ cut nodes that will be added either before or immediately after the first branch node?}
    \label{fig:HeavyCuts}
\end{figure}
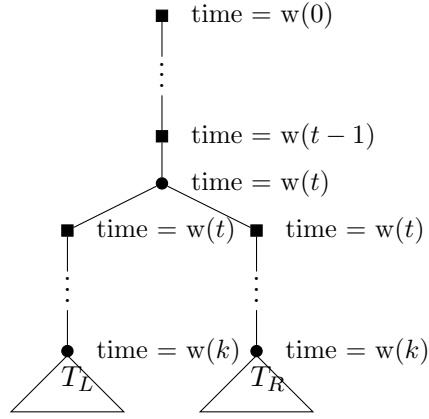
}

\begin{lemma}\label{lem:heavy}
Consider a \bc tree in which each root-to-leaf path has exactly $k$ cut nodes, and each cut node can only be located either before or immediately after the first branching node.
Then the \heavycutweight of the tree is minimized by adding
\begin{equation*}
    \cutind^\star
    \in \argmin_{0 \le \cutind \leq k} \left\{
      \HCf(t) - \sum_{i=0}^{\cutind-1} \HCf(i)
    \right \}
\end{equation*}
cut nodes before the first branch node, and $k-\cutind^\star$ cut nodes in a path starting at each child of the first branch node.
\end{lemma}
\begin{proof}
Suppose tree $T$ has $\cutind$ cut nodes at the root, followed by a branch node, then $k-\cutind$ cut nodes at each child of the branch node, followed by subtrees $T_L$ and $T_R$ in the left and right child;
refer to \Cref{fig:HeavyCuts}.
Then, the tree time is
\begin{equation*}
    \treeweight(T) = 
        \sum_{i=0}^\cutind \HCf(i) 
        + 2\left( \sum_{i=\cutind}^k \HCf(i) \right)
        + \treeweight(T_L) + \treeweight(T_R).
\end{equation*}
The first sum corresponds to the time of the nodes before branching, %
the $\cutind$ cut nodes and 1 branch node.
The next term is the time of the $k-\cutind$ cut nodes added after the first branch node, for each branch.
Finally, we add the times of the remaining subtrees.

We are interested in finding $\cutind$ that minimizes the tree's \heavycutweight. Hence,
\begin{align*}
    \cutind^\star &\in \argmin_{0 \le \cutind \le k} \left\{
        \sum_{i=0}^\cutind \HCf(i) 
        + 2 \sum_{i=\cutind}^k \HCf(i)
        + \treeweight(T_L) + \treeweight(T_R)
    \right\}
    \\ &=\argmin_{0 \le \cutind \le k} \left\{
        \sum_{i=0}^k \HCf(i) 
        + \HCf(\cutind) + \sum_{i=\cutind}^k \HCf(i)
    \right\}
    \\ &=\argmin_{0 \le \cutind \le k} \left\{
        \HCf(\cutind) + \sum_{i=\cutind}^k \HCf(i)
    \right\}
    \\ &=\argmin_{0 \le \cutind \le k} \left\{
        \HCf(\cutind) + \sum_{i=0}^k \HCf(i) - \sum_{i=0}^{\cutind-1} \HCf(i)
    \right\}
    \\ &=\argmin_{0 \le \cutind \le k} \left\{
        \HCf(\cutind) - \sum_{i=0}^{\cutind-1} \HCf(i)
    \right\}. 
    \qedhere
\end{align*}
\end{proof}

\begin{lemma}
\label{lem:cannot-move-cuts-up}
    If $\lbd=\rbd$, and $T$ is a symmetric $\treeweight$-minimal tree proving bound $\targetbound$ with a path of $k$ cut nodes incident to each child of the root node, then for all $q \in [1,k]$, it holds that
      \begin{equation}
          \HCf(q) \ge 1 + \sum_{i=0}^{q-1} \HCf(i).
      \end{equation}
\end{lemma}
\begin{proof}
    Applying \Cref{lem:heavy}, the minimality of $T$ implies that placing the $k$ cuts after the root node is weakly better than shifting any number $q \in [1,k]$ cut nodes to the root.
    In other words, the minimizer in \Cref{lem:heavy} is $\cutind^\star = 0$, which implies that
        $\HCf(q) - \sum_{i=0}^{q-1} \HCf(i) \ge \HCf(0) = 1$
    for any $q \in [1,k]$.
\end{proof}

\subsection{Proof of \texorpdfstring{\Cref{thm:root-cuts-suffice}}{Theorem}}
\label{sec:proof-of-thm}

\begin{proof}[Proof of \Cref{thm:root-cuts-suffice}]
    Let $\distance_T(u)$ denote the length of the path in $T$ from the root to node $u$.
    Define $\lengthcutpath_T(u)$ as the number of cut nodes in the subtree rooted at node $u$.
    For a tree $T$, denote the deepest branch node in $T$ that has cut nodes as descendants by
        \[
            \deepestcutnodestart(T) 
            \in 
            \argmax_{
                u \in T
            }
                \{\distance_T(u) \suchthat \lengthcutpath_T(u) > 0,\, u \text{ branch node}\},
        \]
    where we define $\deepestcutnodestart(T)$ as the root of $T$ if there are no cuts or they all form a path at the root.
    
    Let $T$ denote a symmetric (without loss of generality by \cref{lem:l=r_symmetric-tree}) $\treeweight$-minimal $\SVBC{\lrcw}$ tree
    proving bound $\targetbound$
    such that, among all $\treeweight$-minimal trees, $T$ minimizes $\distance_T(\deepestcutnodestart(T))$.
    There is nothing to prove if there are no cut nodes or they all form a path at the root,
    so assume for the sake of contradiction that the cut nodes do not all form a path at the root.
    
    Let $T^\star$ denote the subtree rooted at $u \defeq \deepestcutnodestart(T)$.
    In $T^\star$, $u$ is a branch node, each child of $u$ is a cut node, and after a path of $\lengthcutpath_{T}(u)$ cut nodes from each child, the remainder of the tree is only branch or leaf nodes.
    Note that, by \Cref{lem:subtrees-are-optimal}, $T^\star$ is a $\treeweight$-minimal $\SVBC(\lbd, \rbd; \cbd, \bar{\HCf})$ tree proving bound $\targetbound - \gapfn[T](u)$, 
    where $\bar{\HCf}(z) \defeq \HCf(K+z) / \HCf(K)$ for any $z \in \nonnegints$ and $K$ is the number of cut nodes on the path from the root of $T$ to $u$.
    For convenience, define $k \defeq \lengthcutpath_{T}(u)$, and (without loss of generality) assume $K = 0$, so $\bar{\HCf} = \HCf$.
    Our contradiction will come from proving that $T^\star$ cannot be $\treeweight$-minimal.

    From \Cref{lem:cannot-move-cuts-up} with $q=k$,
    moving the $k$ cuts up to the root node must increase the tree time with respect to $T^\star$:
        \begin{equation}
        \label{eq:move-k-cuts}
            \HCf(k) > 1 + \sum_{i=0}^{k-1} \HCf(i).
        \end{equation}

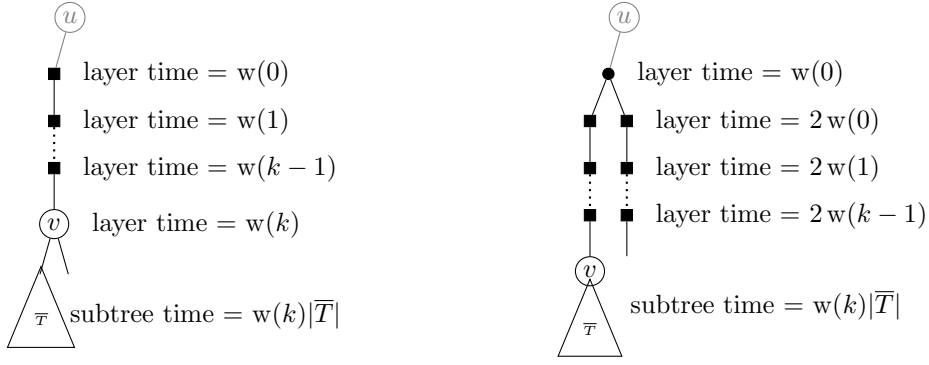
\begin{figure}
    \centering
    \captionsetup[subfigure]{justification=centering}
    \begin{subfigure}[T]{%
        \ifspringer
        0.45\textwidth
        \else
        0.45\textwidth
        \fi
    }
    \centering
    \begin{forest}
        for tree = {small node,l=1pt}
        [$u$,fill=none,inner sep=1.5pt,gray,
            [,cut,lvlwt={$\HCf(0)$},edge={gray}
                [,cut,lvlwt={$\HCf(1)$}
                        [,cut,lvlwt={$\HCf(k-1)$},edge={dotted,thick},
                            [$v$,fill=none,inner sep=1.5pt,lvlwt={$\HCf(k)$},
                                [,empty,tikz={ \node [itria,xshift=1pt,fit to=tree,font=\fontsize{5}{5}\selectfont] {$\bar{T}$}; },subtreewtdeep={$\HCf(k) \cardinality{\bar{T}}$}
                                ]
                                [,empty
                                ]
                            ]
                        ]
                ]
            ]
            [,phantom]
        ]
    \end{forest}
    \end{subfigure}
    \begin{subfigure}[T]{%
        \ifspringer
        0.5\textwidth
        \else
        0.45\textwidth
        \fi
    }
    \centering
    \begin{forest}
        for tree = {small node,l=1pt}
        [$u$,fill=none,inner sep=1.5pt,gray
            [,lvlwt={$\HCf(0)$},edge={gray}
                [,cut,%
                    [,cut,
                            [,cut,edge={dotted,thick},
                                [$v$,fill=none,inner sep=1.5pt,tikz={ \node [itria,xshift=0pt,yshift=-.25cm,inner sep=0pt,fit to=tree,font=\fontsize{5}{5}\selectfont] {$\bar{T}$}; },subtreewtdeep={$\HCf(k) \cardinality{\bar{T}}$}
                                ]
                            ]
                    ]
                ]
                [,cut,lvlwt={$2\HCf(0)$}, %
                    [,cut,lvlwt={$2\HCf(1)$}
                            [,cut,lvlwt={$2\HCf(k-1)$},edge={dotted,thick},
                                [,empty,%
                                ]
                            ]
                    ]
                ]
            ]
            [,phantom]
        ]
    \end{forest}
    \end{subfigure}
    \caption{
        The left panel shows tree $T'$, rooted at the cut node child of $u$,
        while the right panel shows $T''$, the tree obtained from $T'$ by shifting the $k$ cuts down one level,
        where $v$ is now the root of $\bar{T}$.
    }
    \label{fig:move-cuts-down}
\end{figure}

    Let $T'$ denote the subtree rooted at the left child of $u$.
    \Cref{fig:move-cuts-down} depicts $T'$ and a tree $T''$ obtained from $T'$ by shifting the $k$ cuts down a layer.  %
    By \Cref{lem:subtrees-are-optimal}, $T'$ is a $\treeweight$-minimal $\SVBC(\lbd,\rbd; \cbd, \HCf)$ tree proving bound $\targetbound' \defeq \targetbound - \gapfn[T](u) - r$.
    Let $v$ denote the child of the last cut node;
    if $v$ is a branch node, let $\bar{T}$ denote the subtree rooted at either child of $v$,
    and if $v$ is a leaf node, let $\bar{T}$ be empty.
    Then inequality~\eqref{eq:move-k-cuts} implies that %
        \[
            \treetime(T')
            =
          \sum_{i=0}^{k-1} \HCf(i) + \HCf(k) + 2 \HCf(k) \cardinality{\bar{T}}
          >
          1 + 2 \sum_{i=0}^{k-1} \HCf(i) + 2 \HCf(k) \cardinality{\bar{T}}
          = \treetime(T'').
        \]
    The last expression is precisely the time of the new tree $T''$ that proves bound $\targetbound'$, in which the $k$ cuts are shifted down one layer.
    In $T''$, $v$ replaces the root of $\bar{T}$, rather than being the root node's parent as in $T'$,
    with bound 
        $\gapfn[T''](v) 
            = \rbd + k \cbd$;
    all other nodes in $\bar{T}$ have the same bound in both $T'$ and $T''$.
    Note that if $\bar{T}$ is empty, i.e., $v$ is a leaf node,
    then define $T''$ as a tree rooted at a branch node attached to two paths of length $k$, corresponding to the left and right branches consisting of $k-1$ cut nodes and a leaf node.
    When $v$ is a leaf node in $T'$, $\gapfn[T'](v) = k \cbd \ge \targetbound'$, which implies that the leaf nodes of $T''$ have bound $\rbd + (k-1) \cbd \ge \targetbound'$;
    in other words, the ``shift'' operation decreases the total number of cut nodes.
    In either case, the above inequality implies that $\treetime(T') > \treetime(T'')$, contradicting the $\treeweight$-minimality of $T'$ and hence of $T^\star$.
\end{proof}

\section{Conclusion and Potential Extensions}
\label{sec:conclusion}

We analyze a framework capturing several crucial tradeoffs in jointly making branching and cutting decisions for optimization problems.
For example, we show that adding cuts can yield nonmonotonic changes in tree size, which can make it difficult to evaluate the effect of cuts computationally.
Our results highlight challenges for improving cut selection schemes, in terms of their effect on \branchandcutx{} tree size and solution time,
albeit for a simplified setting in which the bound improvement from branching is assumed constant and known, and the bound improvement from cutting is either constant or changing in a specific way.
There do exist contexts in which the relative strength of cuts compared to branching decisions can be approximated, such as by inferring properties for a family of instances, an idea that has seen recent success with machine learning methods applied to integer programming problems \cite{khalil2016learning,khalil2017learning,gasse2019exact,TanAgrFae20,HuaWanLiuZheZhaYuaHaoYuWan22,BerFraHen22,TurKocSerWin22,PauZarKraChaMad22}.
This lends hope to apply our results to improve cut selection criteria for such families of instances and this warrants future computational study, though it is far from straightforward.

This paper focuses on the single-variable version of the abstract \branchandcutx{} model.
Some results extend directly to bounds for a generalization of the model permitting different possible branching variables, by assuming the ``single branching variable'' corresponds to the \emph{best} possible branching variable at every node, but an in-depth treatment of the general case remains open.
Further, an appealing extension of the general \timefunction{}s considered in \Cref{sec:general-cut-time} is to investigate branching on general disjunctions~\cite{Mahajan09,KarCor11,CorLibNan11}, which has been the subject of recent computational study~\cite{YanBolSav20}.

We do not consider some important practical factors, such as interaction with primal heuristics, pruning nodes by infeasibility, or the time it takes to generate cuts.

Finally, most of the results we present in \cref{sec:general-cut-time} for general \timefunction{}s assume that branching on a variable leads to the same bound improvement for both children.
The general situation of unequal and/or nonconstant bound improvements remains open, both regarding the best location of cut nodes and the optimal number of cuts to be added, and merits future theoretical and experimental investigation.

\mbox{}\\
\footnotesize \textbf{Acknowledgements.} The authors thank Andrea Lodi, Canada Excellence Research Chair in Data Science for Real-Time Decision Making, for financial support and creating a collaborative environment that facilitated the interactions that led to this paper, as well as Monash University for supporting Pierre's trip to Montr\'{e}al.

\ifspringer
    \bibliographystyle{plainnat}
\else
    \bibliographystyle{plainnat}
\fi
\bibliography{include/akazachk,include/others}

\appendix
\captionsetup[figure]{font=footnotesize,labelfont=footnotesize,belowskip=-5pt}

\section{Computational Results with Selected MIPLIB Instances}
\label{appendix:more-instances}

\cref{fig:more-instances} shows the linear relaxation bound, predicted bound (using the harmonically-worsening cuts model of \cref{sec:svbhc-def}), linear relaxation resolve time, and cumulative number of Gomory cuts added after up to 100 rounds of cuts have been applied to ten additional instances, using the same computational setup described in \cref{sec:svbhc}.
The same general trends are observed as in the two plots in \cref{fig:two-instances}.
For several instances, such as \instance{air05}, \instance{binkar10}, and \instance{swath3}, the predicted bound --- which is calculated based only on the improvement from the first round of cuts --- is quite close to the actual bound changes after tens of rounds.
The prediction tends to be inaccurate (a large overestimate) as more significant tailing in bound improvement occurs, but occasionally underestimates the bound improvement, such as for \instance{eil33-2}.

\ifarxiv
\begin{figure}
    \centering
	\vspace{-1cm}
    \begin{subfigure}[b]{.49\textwidth}
		\centering
        \includegraphics[height=.175\textheight]{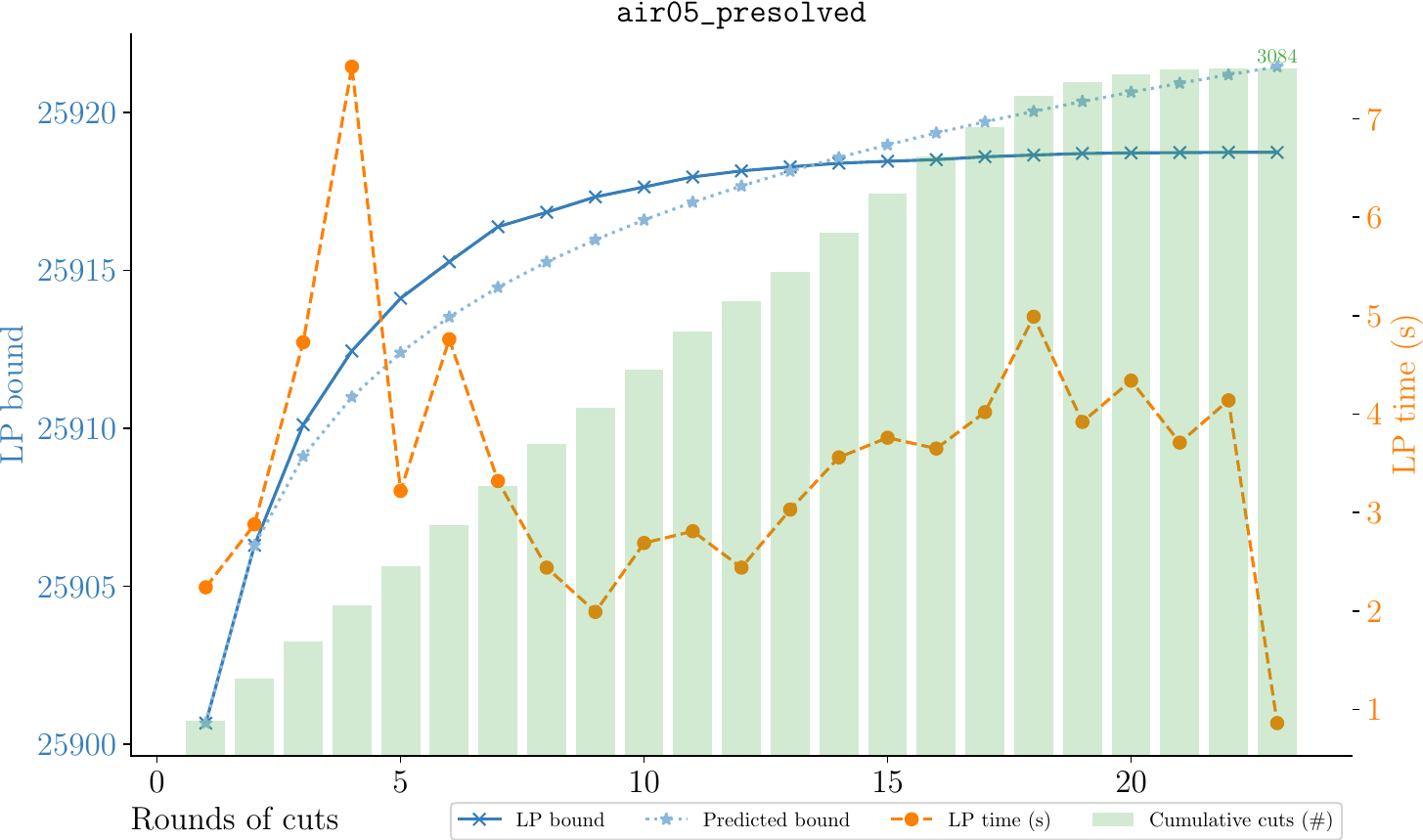}
    \end{subfigure}
    \begin{subfigure}[b]{.49\textwidth}
		\centering
        \includegraphics[height=.175\textheight]{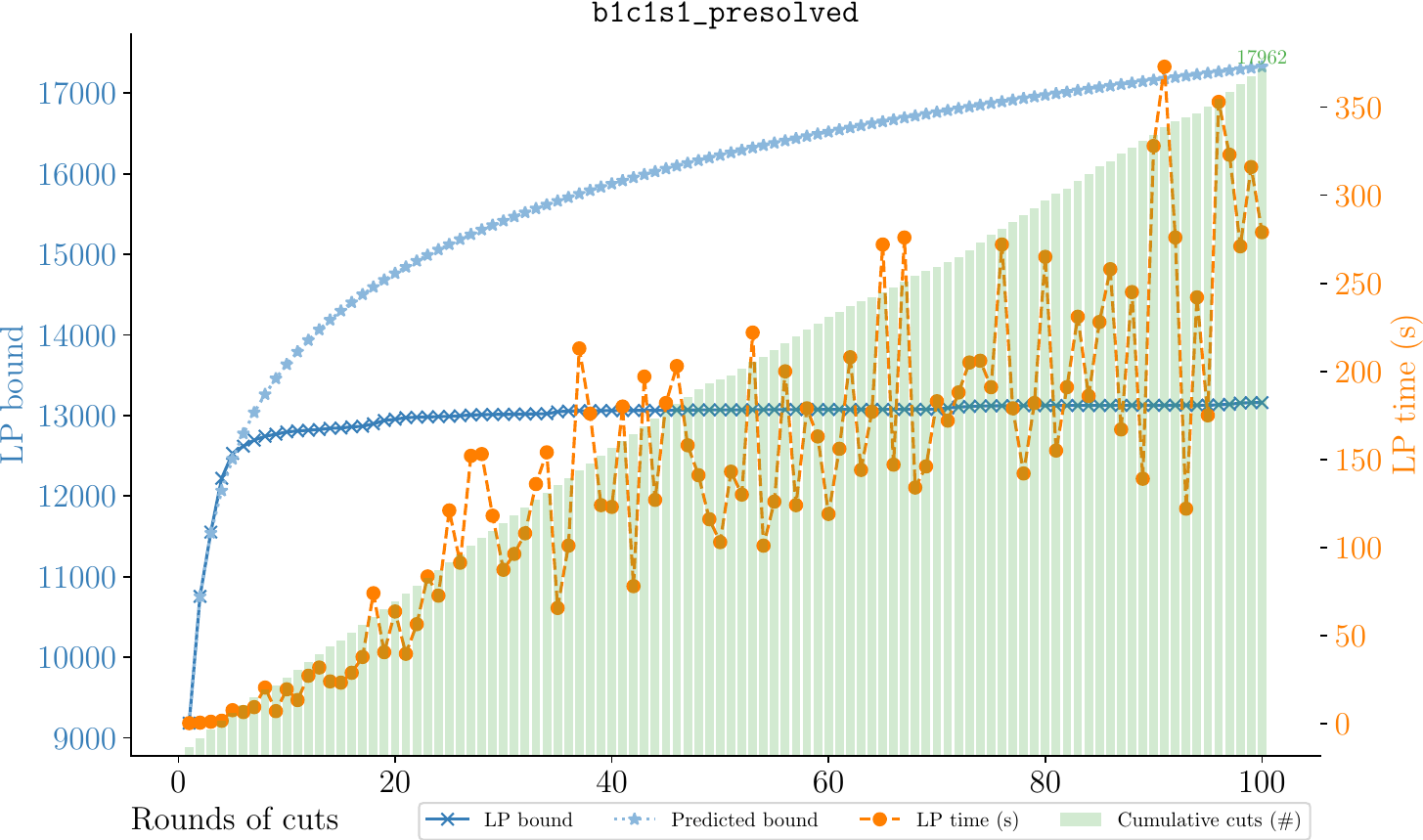}
    \end{subfigure}
    \\\vspace{1em}
    \begin{subfigure}[b]{.49\textwidth}
		\centering
        \includegraphics[height=.175\textheight]{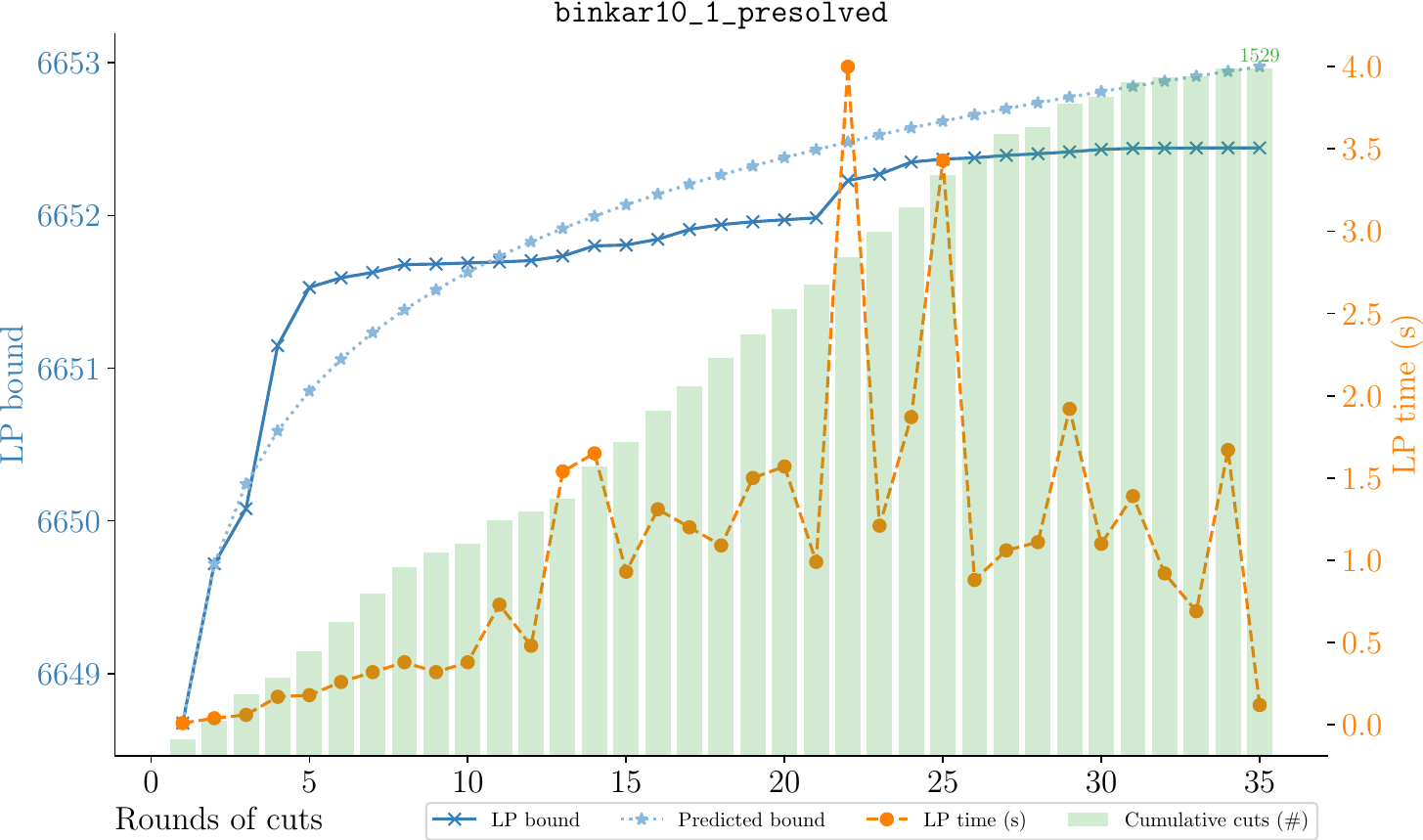}
    \end{subfigure}
    \begin{subfigure}[b]{.49\textwidth}
		\centering
        \includegraphics[height=.175\textheight]{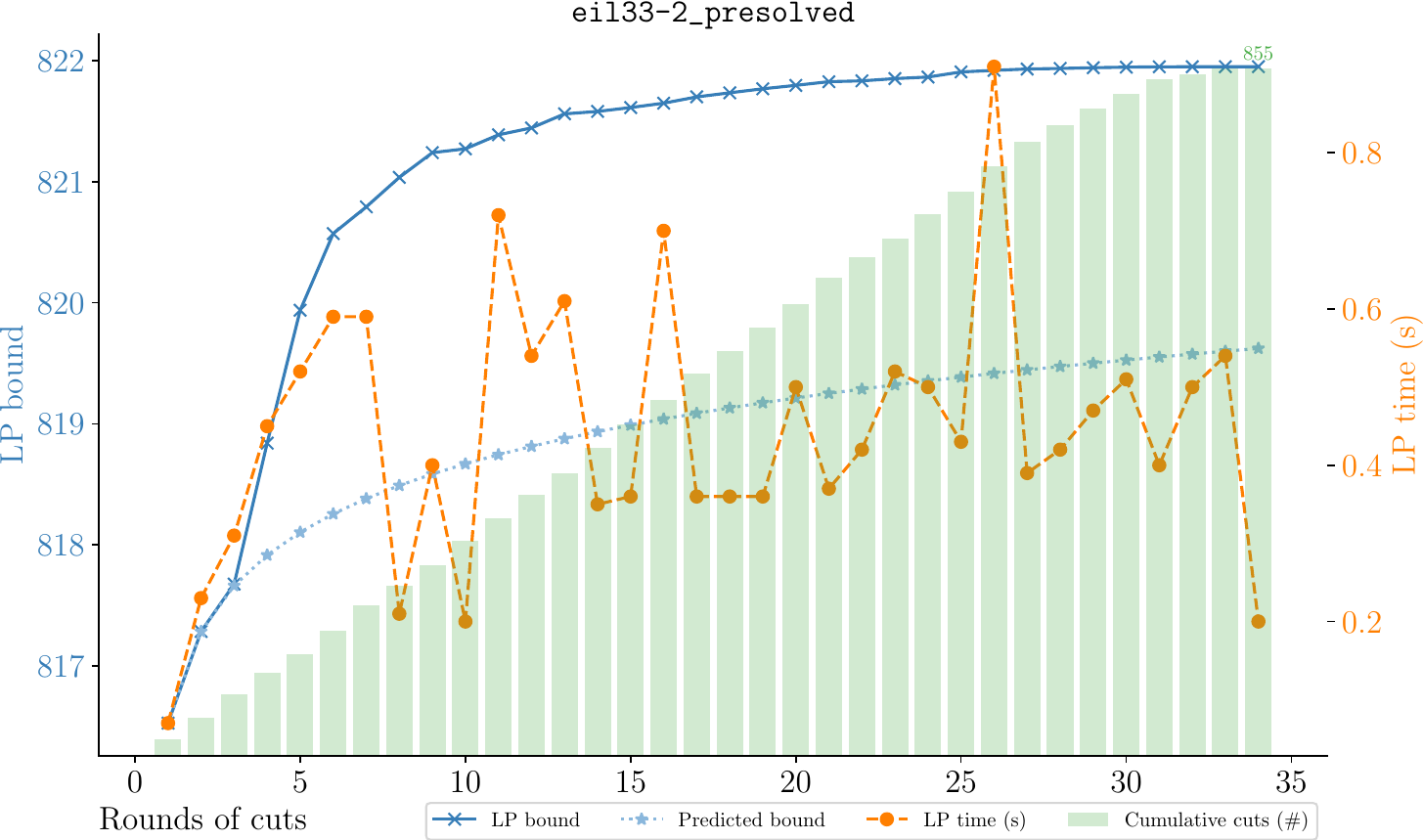}
    \end{subfigure}
    \\\vspace{1em}
    \begin{subfigure}[b]{.49\textwidth}
		\centering
        \includegraphics[height=.175\textheight]{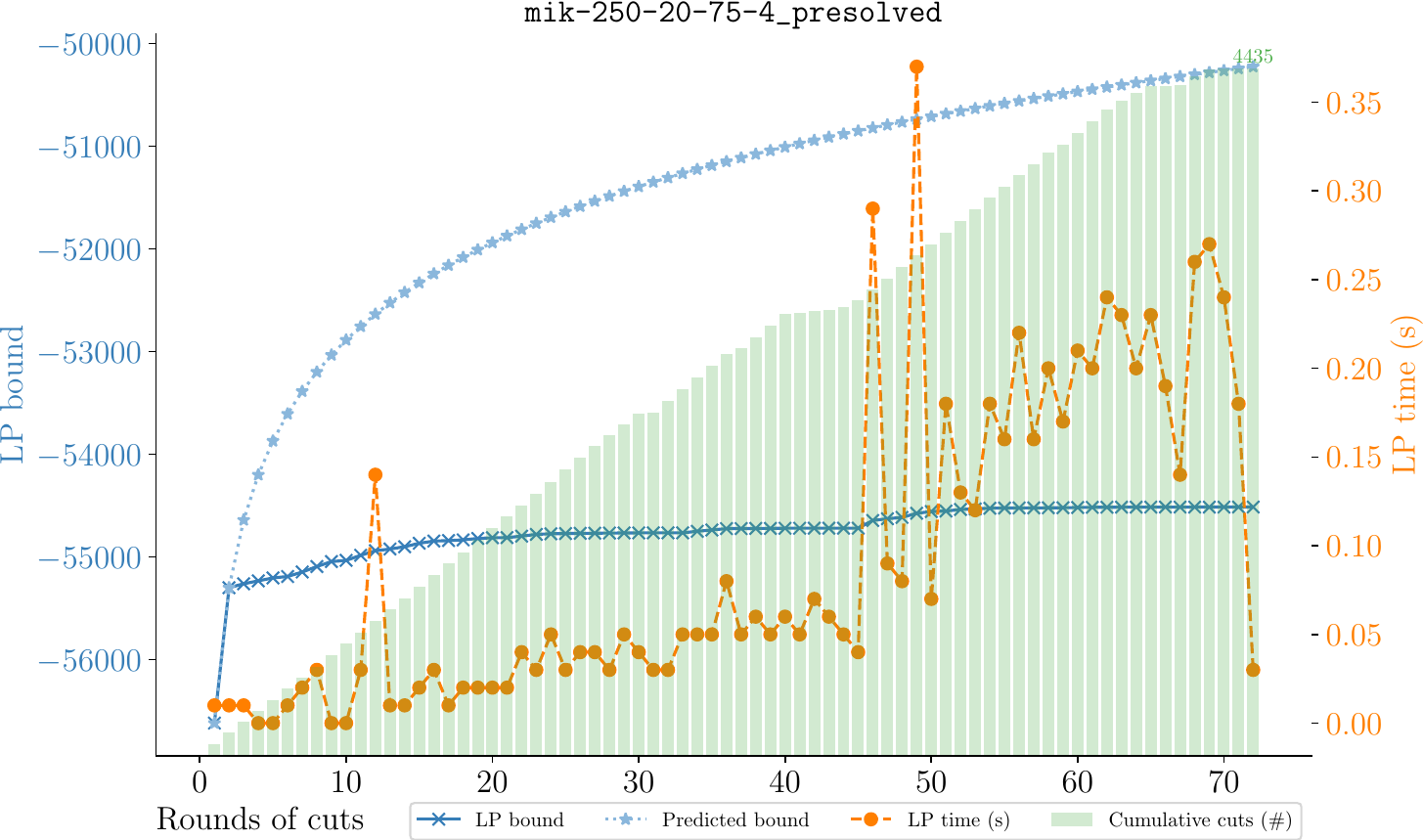}
    \end{subfigure}
    \begin{subfigure}[b]{.49\textwidth}
		\centering
        \includegraphics[height=.175\textheight]{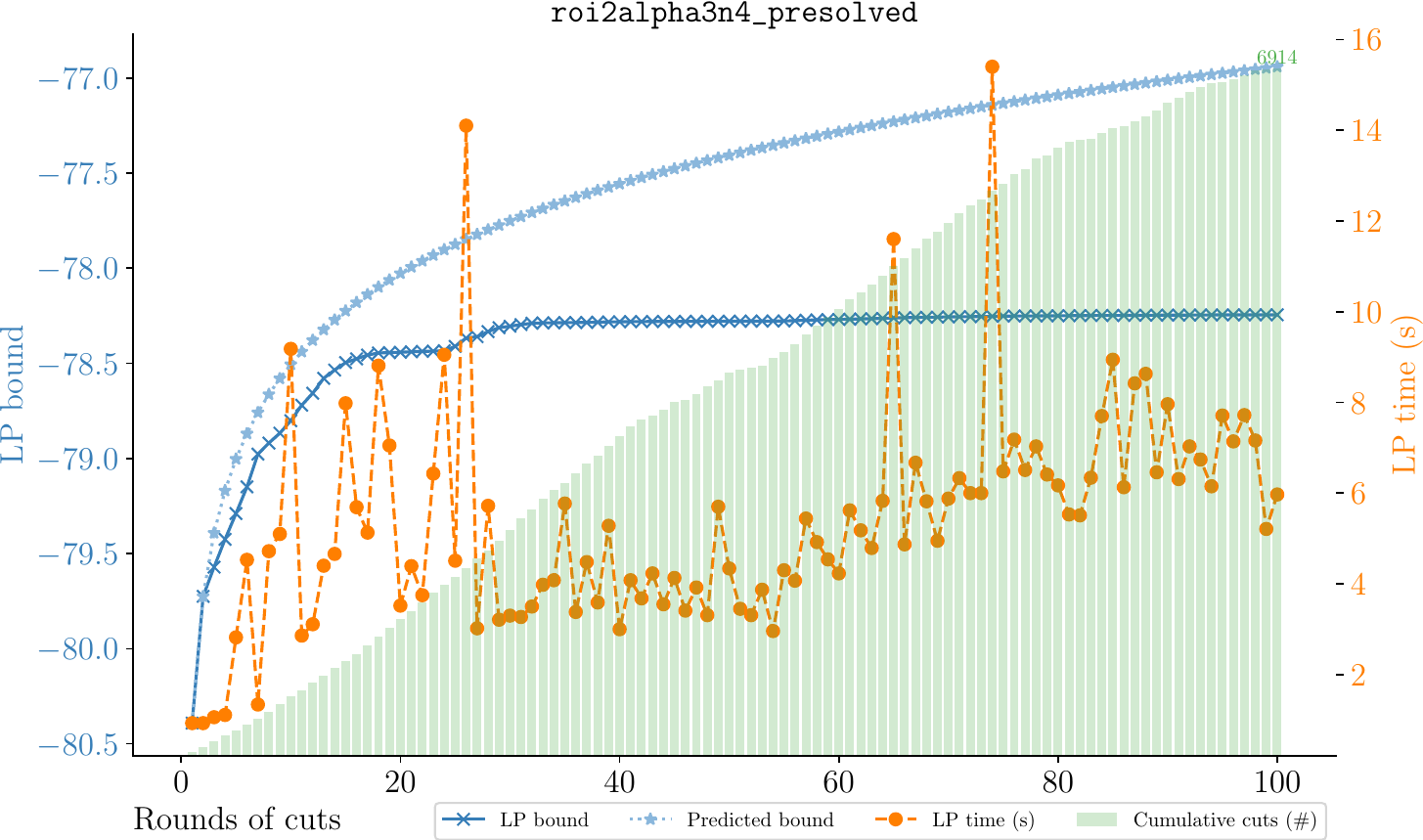}
    \end{subfigure}
    \\\vspace{1em}
    \begin{subfigure}[b]{.49\textwidth}
		\centering
        \includegraphics[height=.175\textheight]{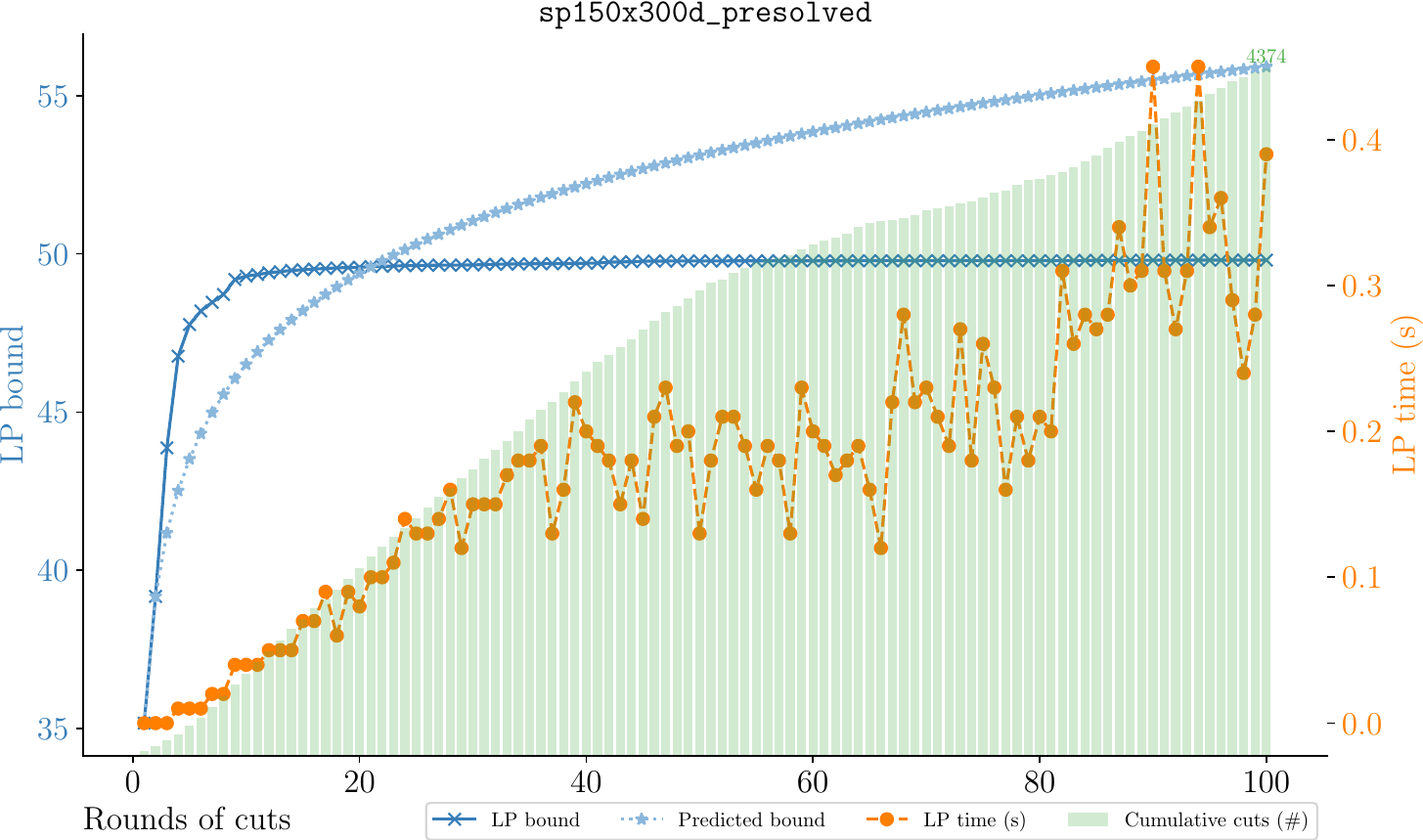}
    \end{subfigure}
    \begin{subfigure}[b]{.49\textwidth}
		\centering
        \includegraphics[height=.175\textheight]{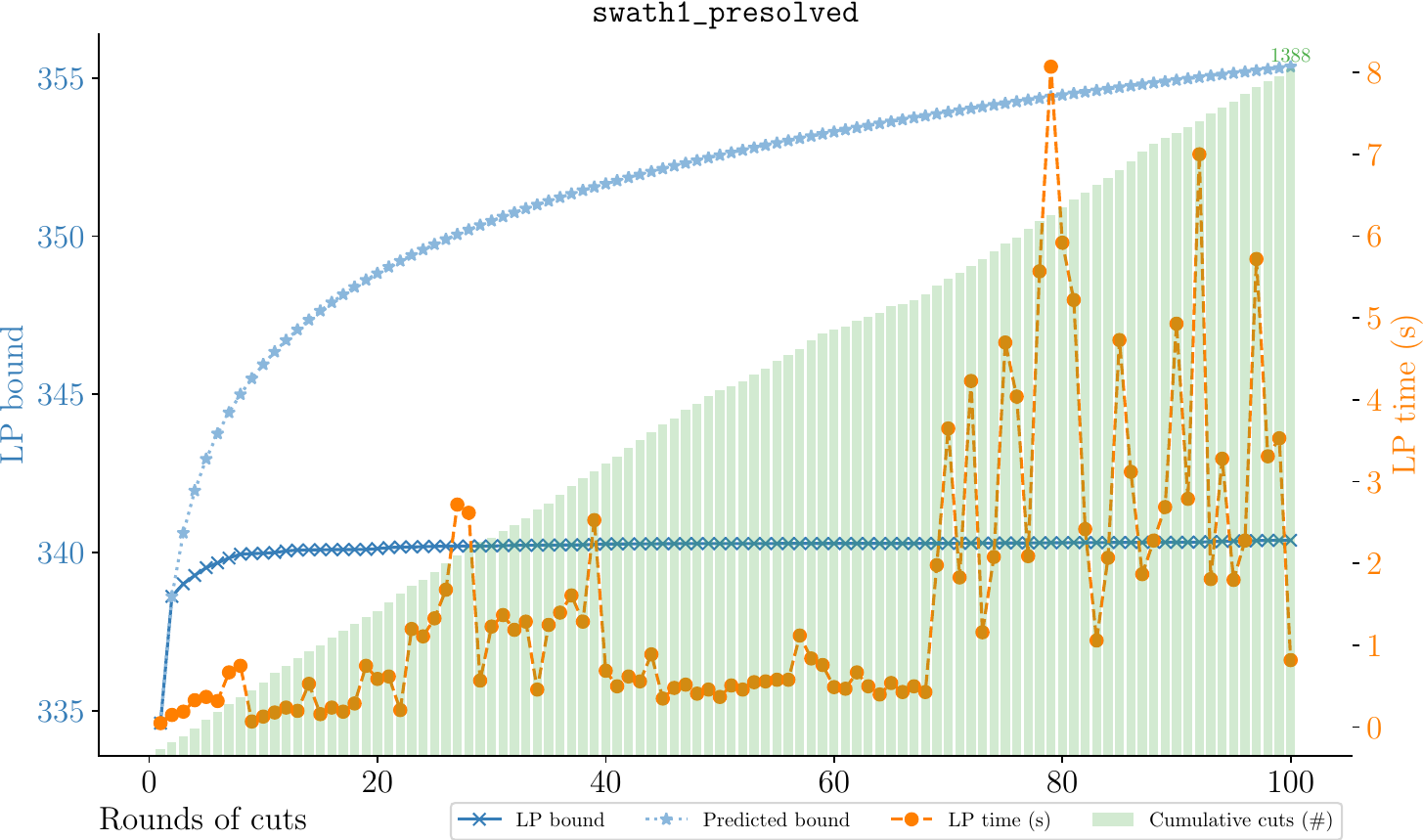}
    \end{subfigure}
    \\\vspace{1em}
    \begin{subfigure}[b]{.49\textwidth}
		\centering
        \includegraphics[height=.175\textheight]{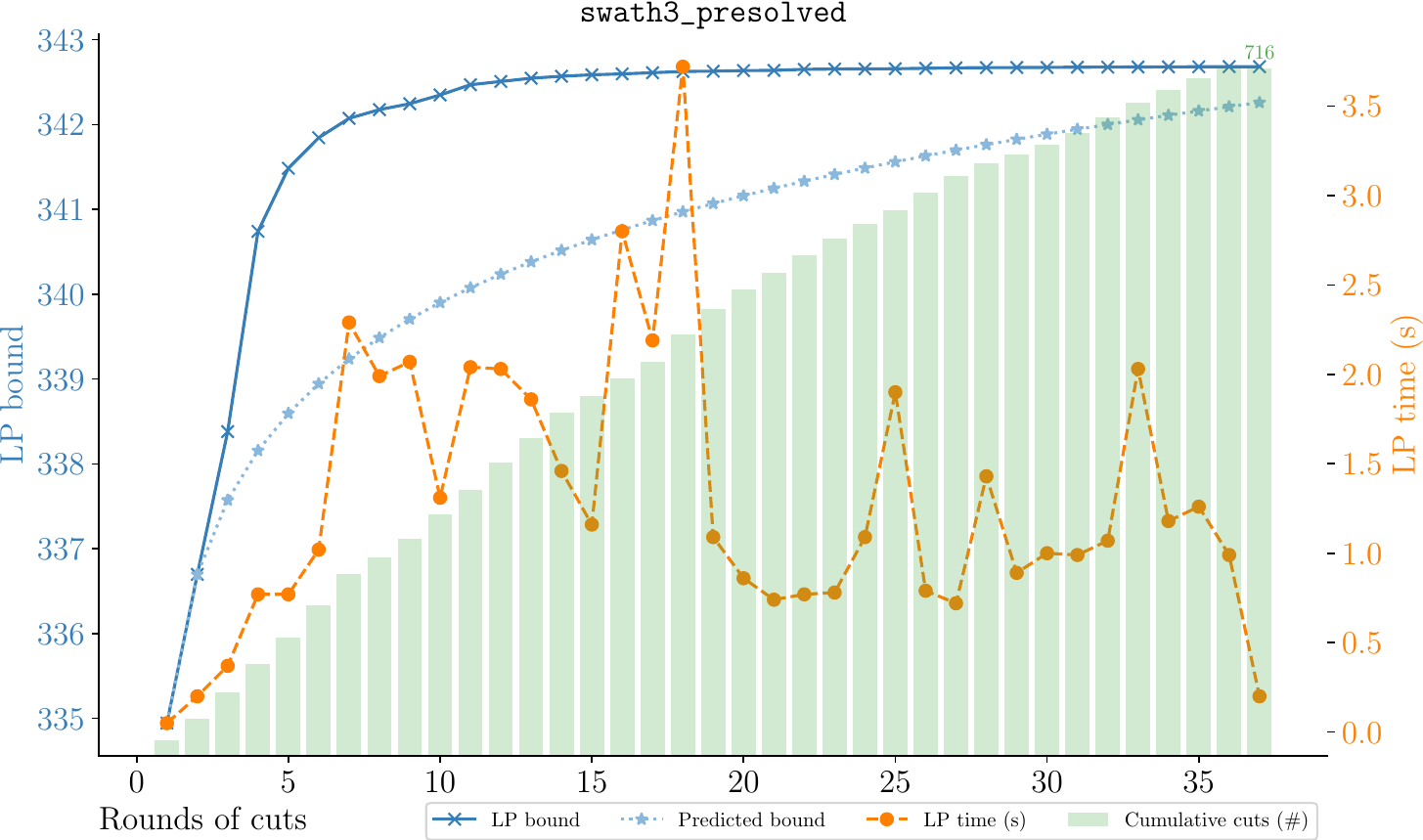}
    \end{subfigure}
    \begin{subfigure}[b]{.49\textwidth}
		\centering
        \includegraphics[height=.175\textheight]{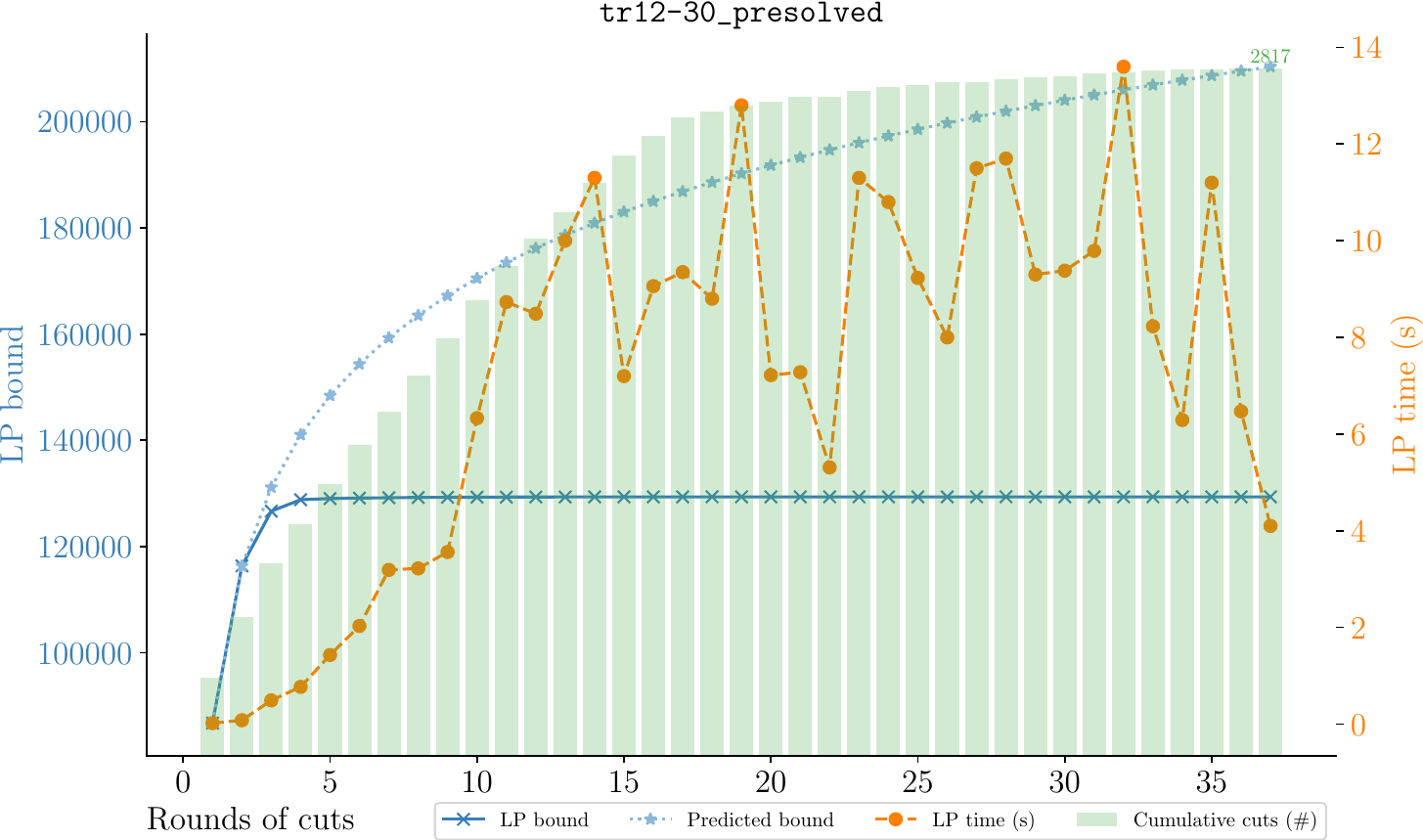}
    \end{subfigure}
    \caption{%
        Applying rounds of cuts on assorted MIPLIB 2017 instances (after preprocessing) typically yields diminishing bound improvement. Several instances show a linear tendency in LP resolve time. The overlayed bar plot for each instance shows the cumulative number of cuts added after each round. The predicted bound using the improvement from the first round follows a logarithmic function that is similar to the actual bound evolution until cut strength exhibits more pronounced tailing off.
    }
    \label{fig:more-instances}
\end{figure}
\else
\begin{figure}
    \centering
    \begin{subfigure}[b]{.49\textwidth}
        \includegraphics[width=\linewidth]{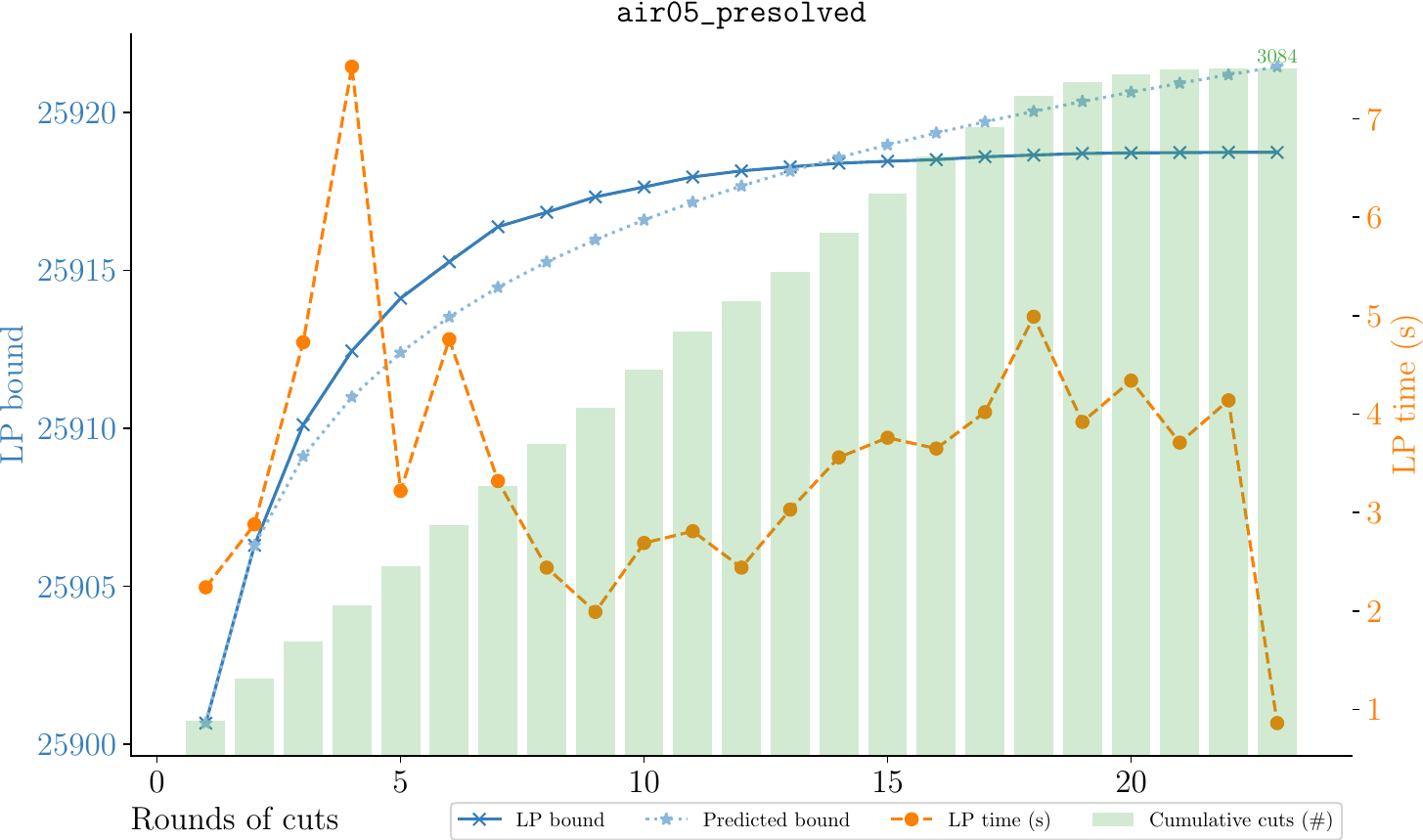}
    \end{subfigure}
    \begin{subfigure}[b]{.49\textwidth}
        \includegraphics[width=\linewidth]{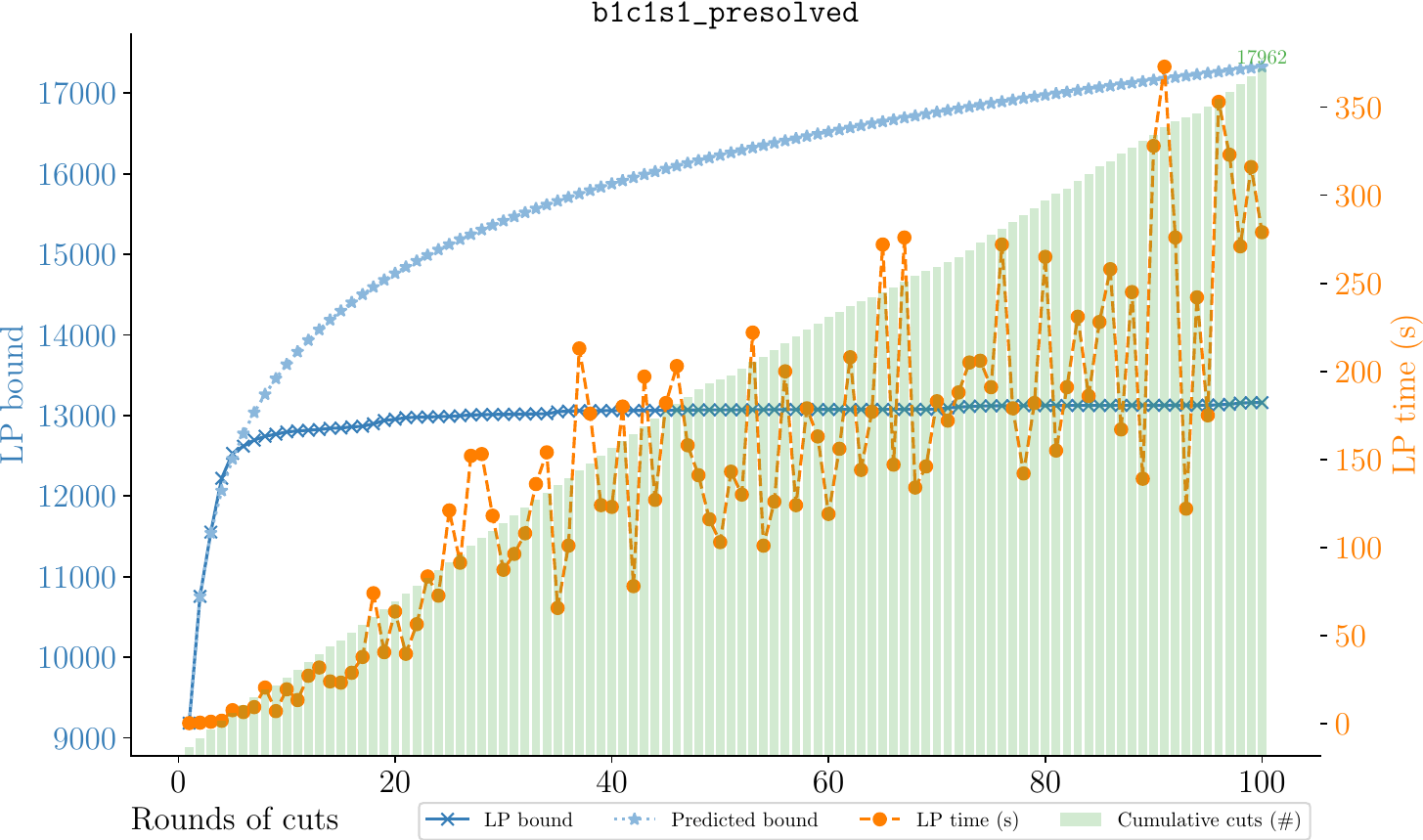}
    \end{subfigure}
    \\\vspace{1em}
    \begin{subfigure}[b]{.49\textwidth}
        \includegraphics[width=\linewidth]{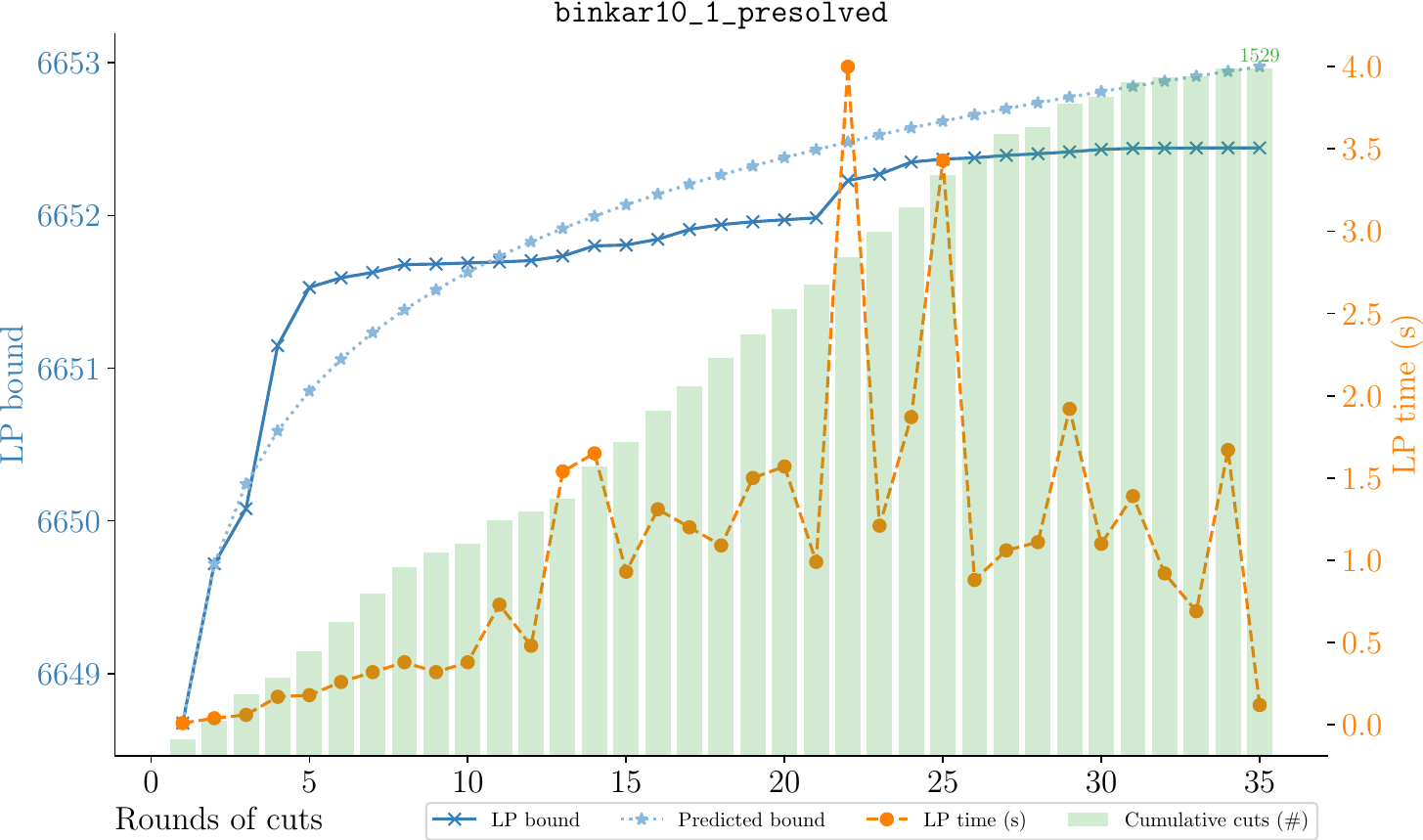}
    \end{subfigure}
    \begin{subfigure}[b]{.49\textwidth}
        \includegraphics[width=\linewidth]{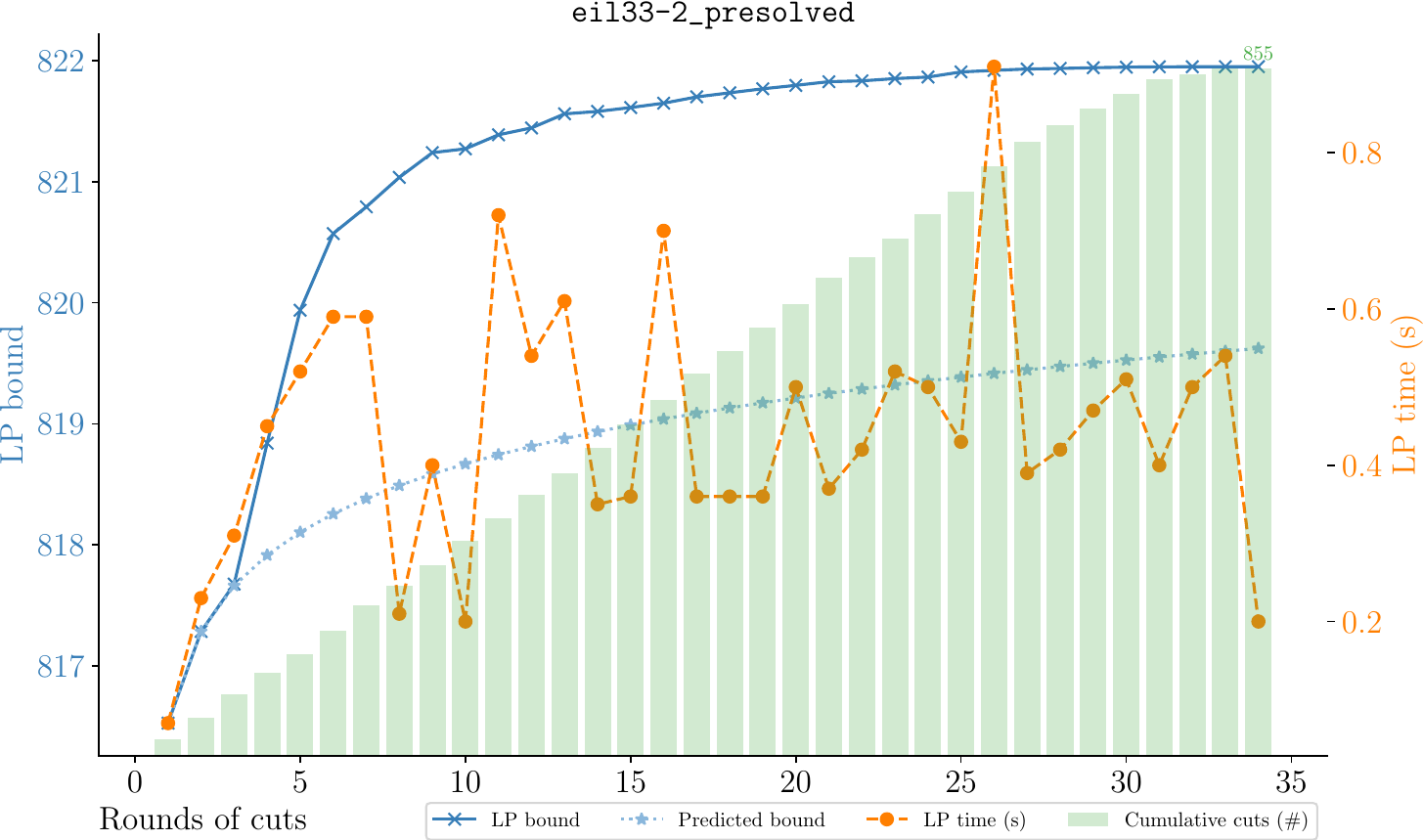}
    \end{subfigure}
    \\\vspace{1em}
    \begin{subfigure}[b]{.49\textwidth}
        \includegraphics[width=\linewidth]{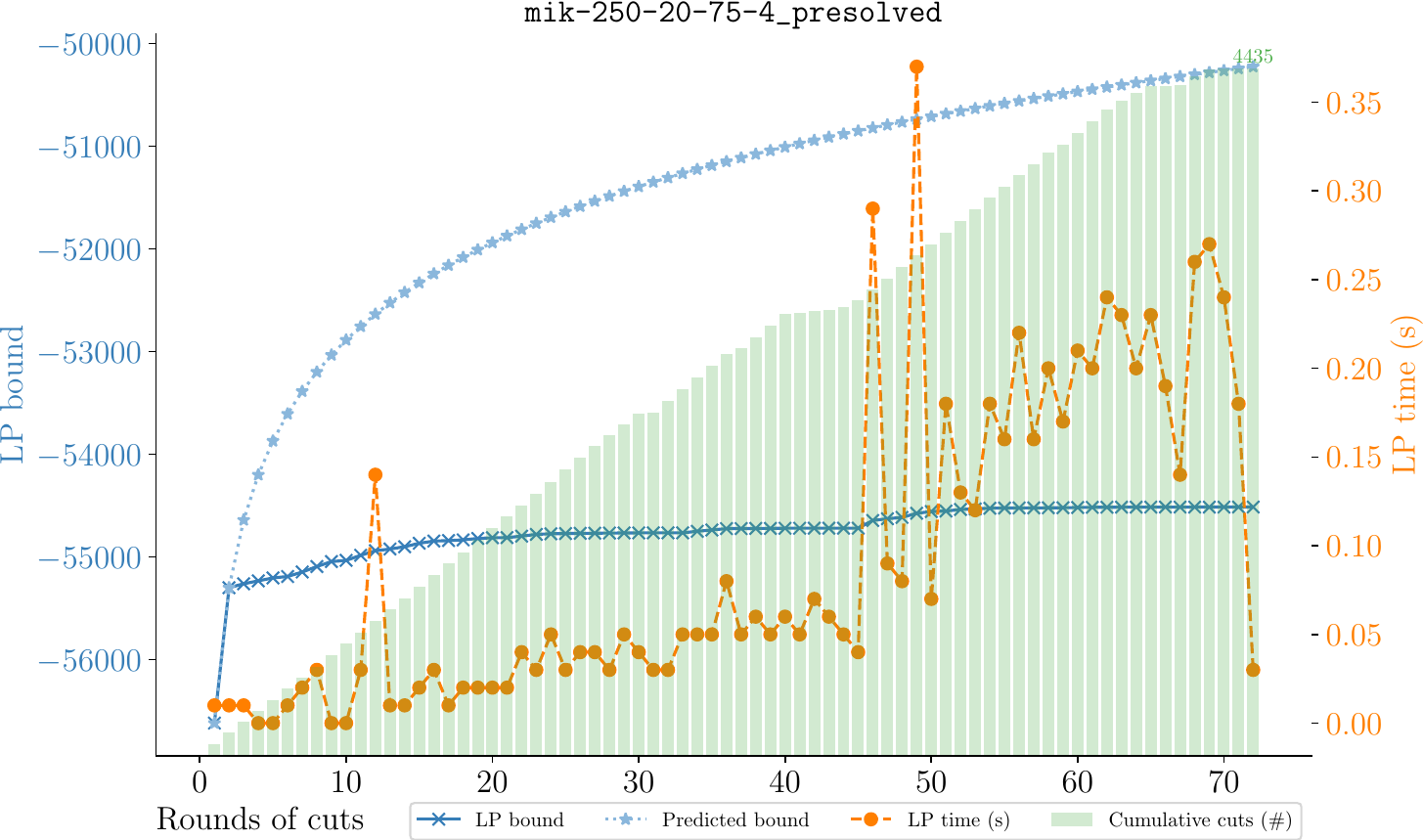}
    \end{subfigure}
    \begin{subfigure}[b]{.49\textwidth}
        \includegraphics[width=\linewidth]{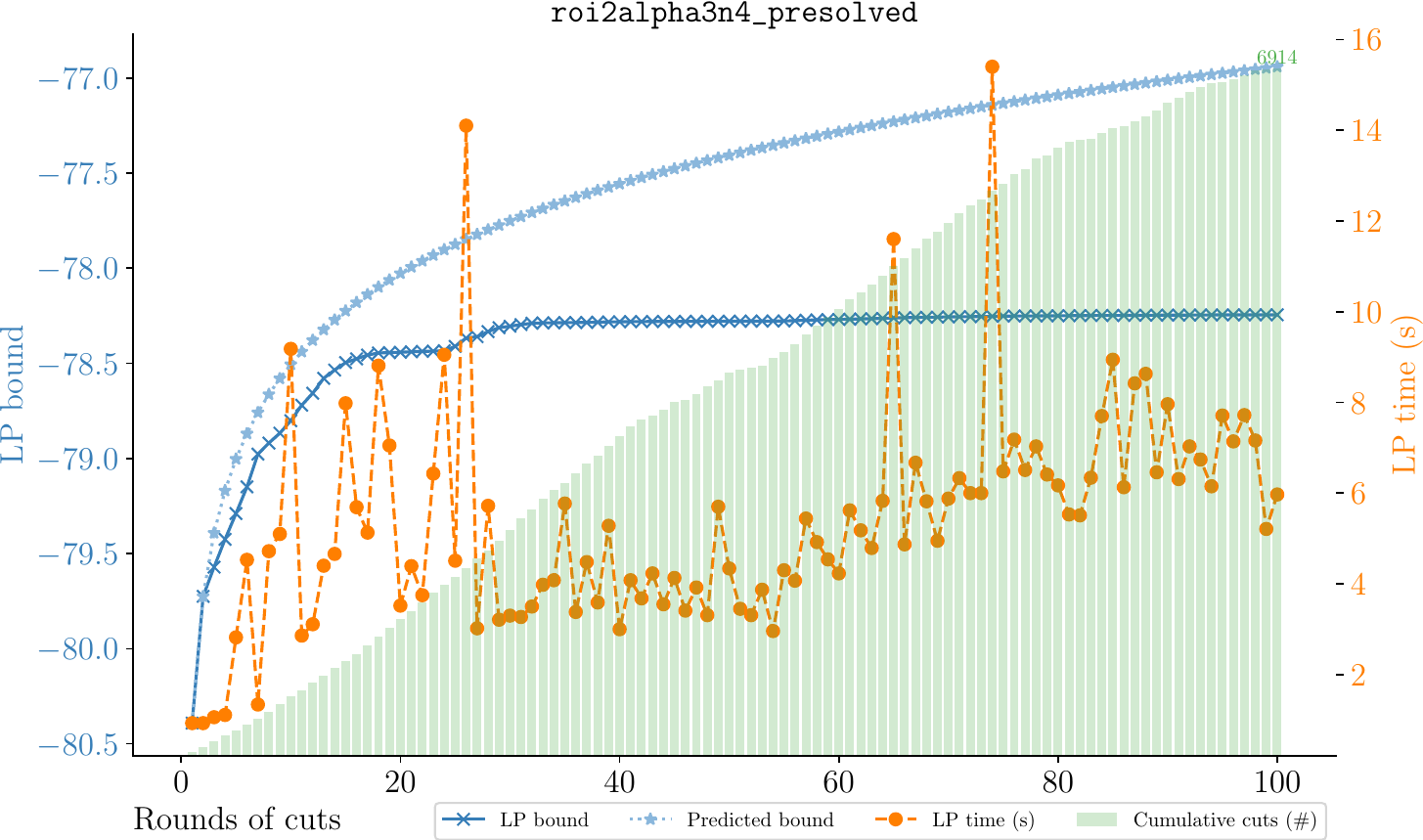}
    \end{subfigure}
    \\\vspace{1em}
    \begin{subfigure}[b]{.49\textwidth}
        \includegraphics[width=\linewidth]{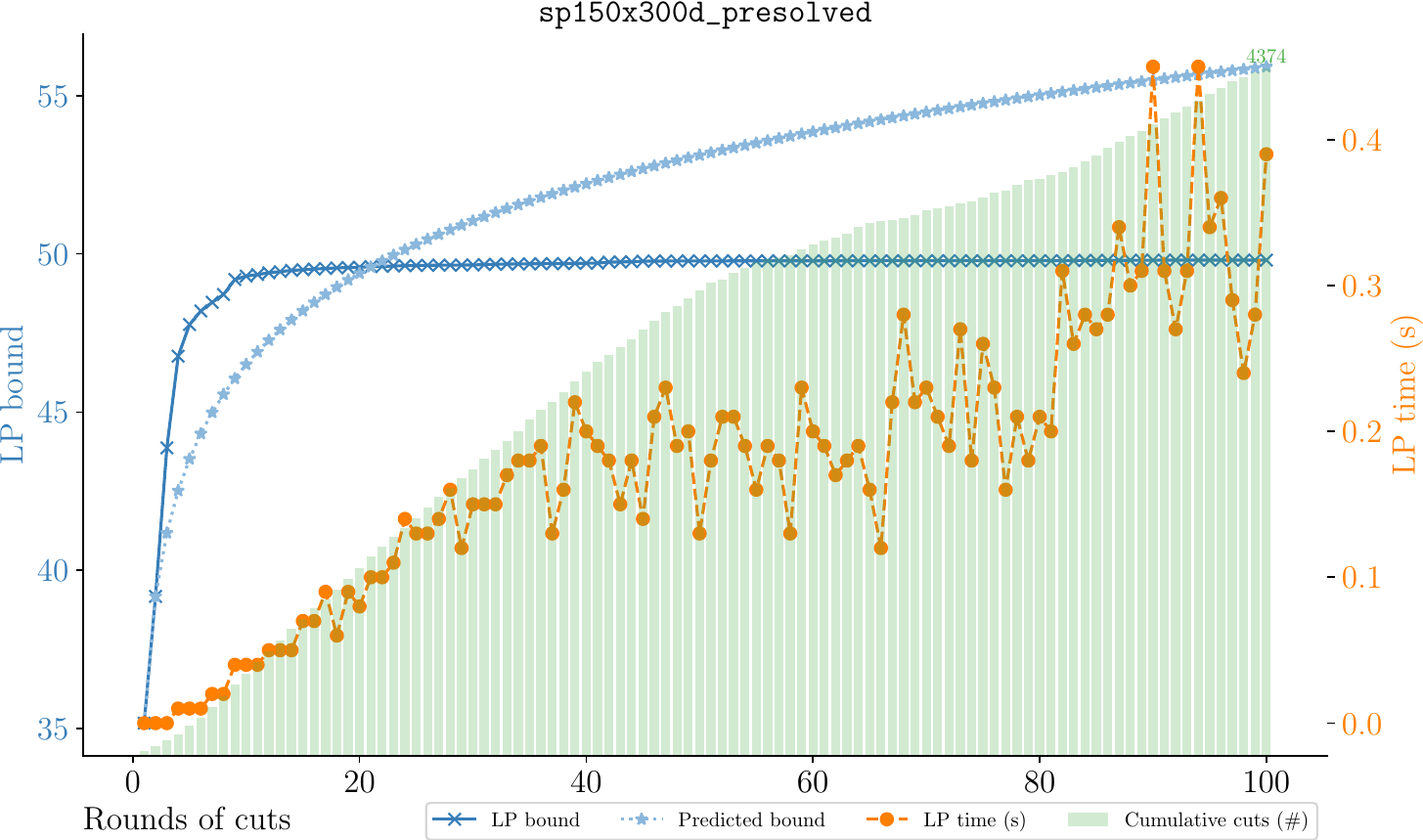}
    \end{subfigure}
    \begin{subfigure}[b]{.49\textwidth}
        \includegraphics[width=\linewidth]{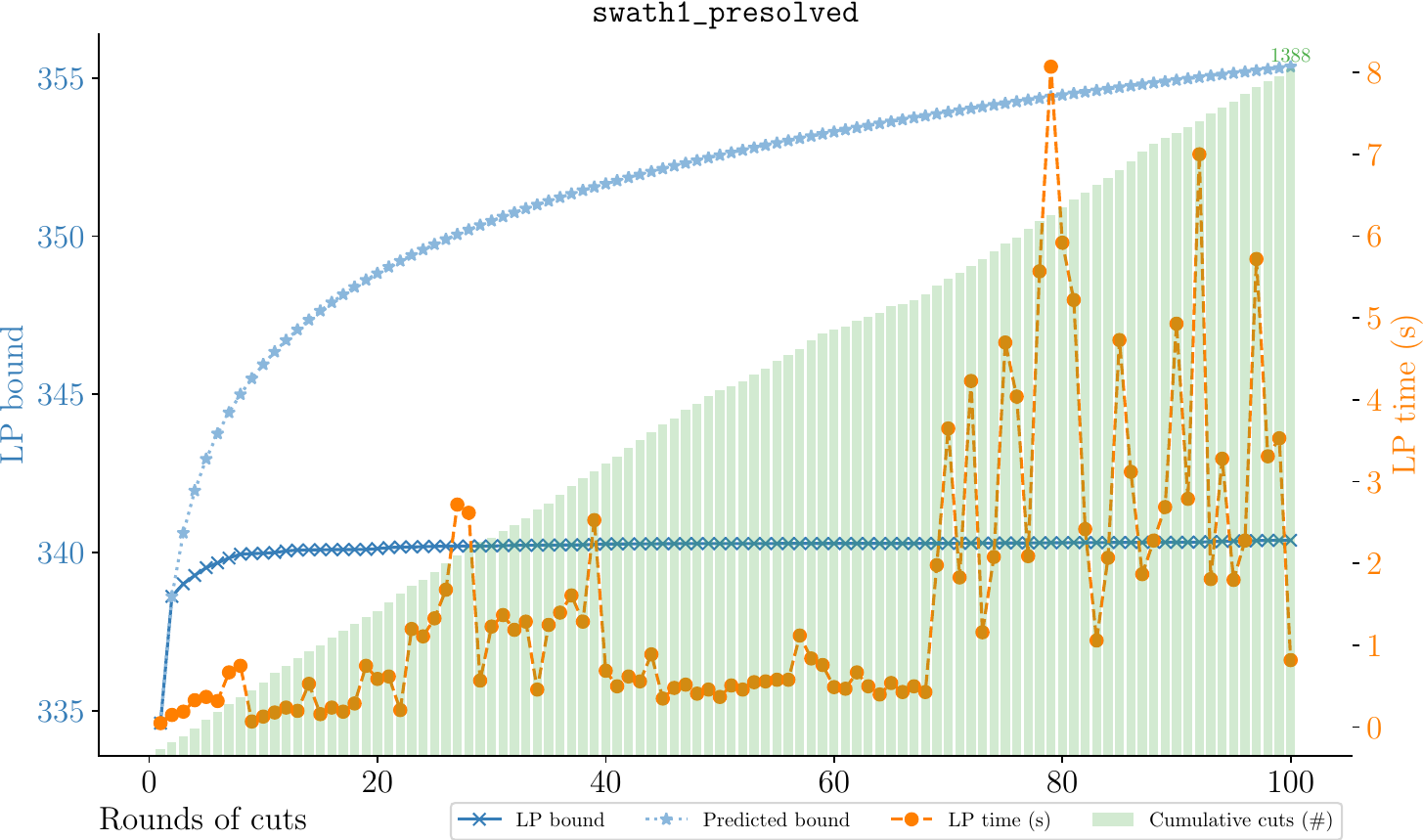}
    \end{subfigure}
    \\\vspace{1em}
    \begin{subfigure}[b]{.49\textwidth}
        \includegraphics[width=\linewidth]{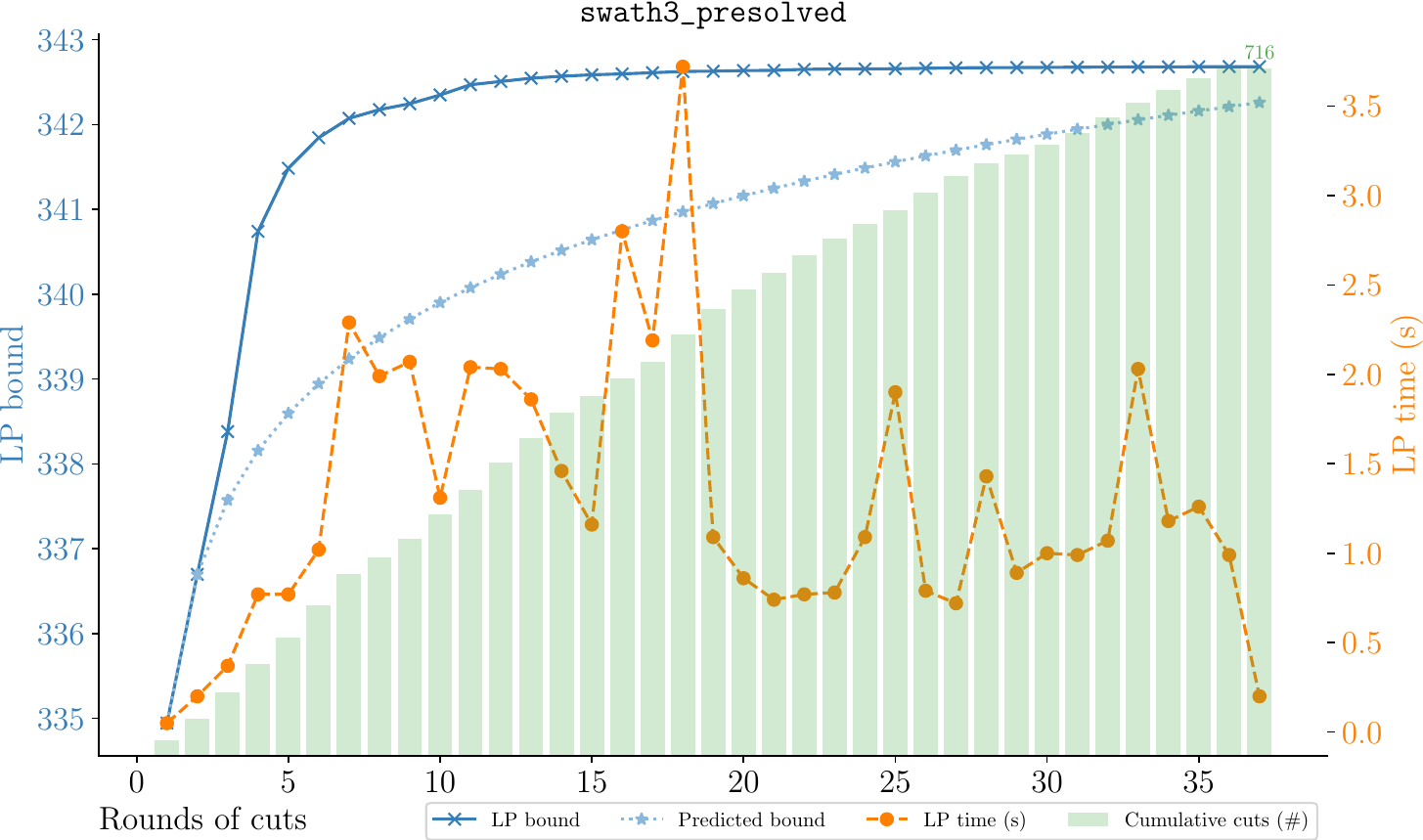}
    \end{subfigure}
    \begin{subfigure}[b]{.49\textwidth}
        \includegraphics[width=\linewidth]{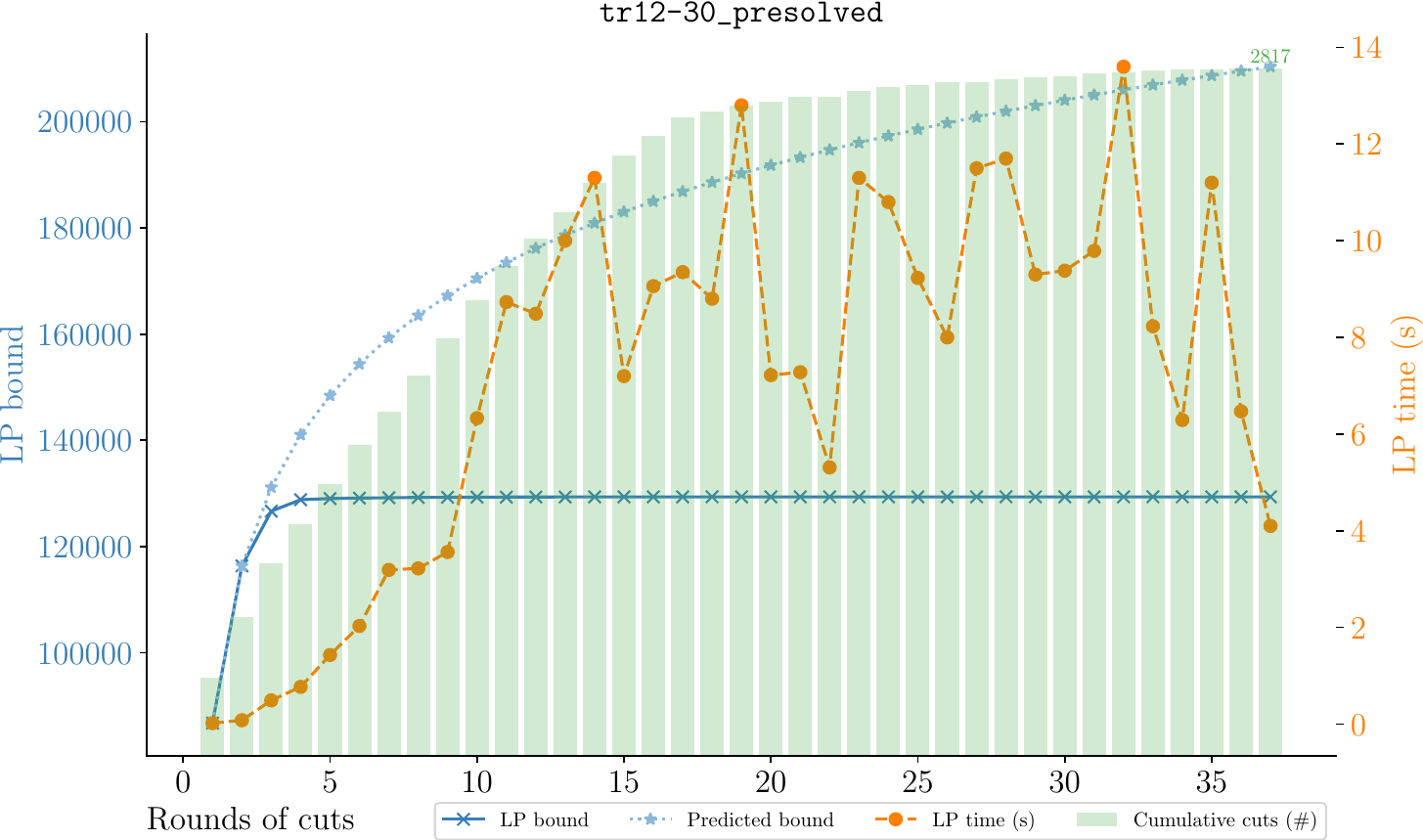}
    \end{subfigure}
    \caption{%
        Applying rounds of cuts on assorted MIPLIB 2017 instances (after preprocessing) typically yields diminishing bound improvement. Several instances show a linear tendency in LP resolve time. The overlayed bar plot for each instance shows the cumulative number of cuts added after each round. The predicted bound using the improvement from the first round follows a logarithmic function that is similar to the actual bound evolution until cut strength exhibits more pronounced tailing off.
    }
    \label{fig:more-instances}
\end{figure}
\fi

\section{Experiments with Optimal Proportion of Cut Rounds in \SVBWC{}}
\label{appendix:cut-rounds-diminishing-cuts}

\cref{thm:ApxMAIN} proves that, in the \SVBHC{} model of \Cref{sec:svbhc}, the number of cuts prescribed by \cref{alg:SVBWC} is approximately optimal in the sense that the resulting tree is at most a multiplicative factor larger than the optimal tree size.
\cref{thm:apxGrowth} shows that using this approximately-optimal number of cuts proves a constant proportion of the overall bound, in the limit when the target bound goes to infinity.
However, since the multiplicative factor in \cref{thm:ApxMAIN} may be quite large, it is not clear if the same type of limit exists for \emph{minimal-size} trees.
In \cref{fig:fracGapClosed}, we address this question computationally, showing that the proportion of bound proved by cut nodes tends to the same limit in a minimal tree for four artificial instances of the \SVBHC{} model.
Experiments with more instances have shown the same behavior and therefore are omitted.

\begin{figure}
    \centering
    \ifspringer
	    \includegraphics[width=.8\textwidth]{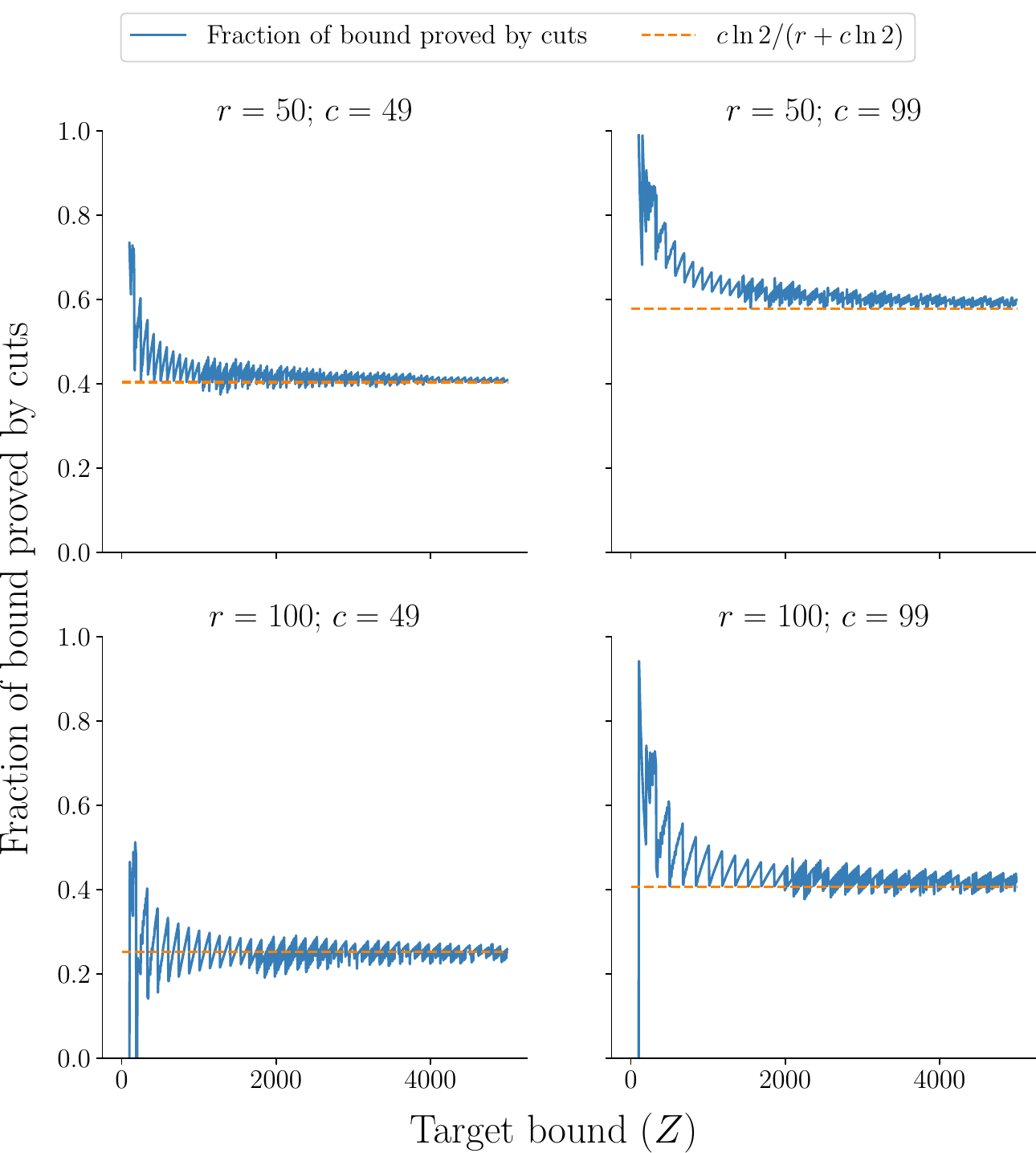}
	\else
		\includegraphics[width=.75\textwidth]{fig-simul/ExactGapClosed.pdf}
	\fi
    \caption{Fraction of bound proved by cut nodes in an (\emph{exactly}-)optimal \SVBWC{} tree, 
    exhibiting convergence to the bound from \cref{thm:apxGrowth} provided by
    the \emph{approximately-optimal} number of cut nodes prescribed by \cref{alg:SVBWC}.
    }
    \label{fig:fracGapClosed}
\end{figure}

\end{document}